\documentclass[11pt]{article}

\usepackage{amssymb}
\usepackage{amsmath}
\usepackage{mathtools}
\usepackage{algpseudocode}
\usepackage{graphics,epsfig}
\DeclareGraphicsExtensions{.pdf}
\usepackage{sidecap}
\usepackage{url,here}
\usepackage[backref,colorlinks=true]{hyperref}
\usepackage{multirow}
\usepackage{algorithm}
\allowdisplaybreaks

\begin{document}
\newenvironment {proof}{{\noindent\bf Proof.}}{\hfill $\Box$ \medskip}

\newtheorem{theorem}{Theorem}[section]
\newtheorem{lemma}[theorem]{Lemma}
\newtheorem{condition}[theorem]{Condition}
\newtheorem{proposition}[theorem]{Proposition}
\newtheorem{remark}[theorem]{Remark}
\newtheorem{definition}[theorem]{Definition}
\newtheorem{hypothesis}[theorem]{Hypothesis}
\newtheorem{corollary}[theorem]{Corollary}
\newtheorem{example}[theorem]{Example}
\newtheorem{descript}[theorem]{Description}
\newtheorem{assumption}[theorem]{Assumption}

\newcommand{\ba}{\begin{align}}
\newcommand{\ea}{\end{align}}

\def\P{\mathbb{P}}
\def\R{\mathbb{R}}
\def\E{\mathbb{E}}
\def\N{\mathbb{N}}
\def\Z{\mathbb{Z}}

\renewcommand {\theequation}{\arabic{section}.\arabic{equation}}
\def \non{{\nonumber}}
\def \hat{\widehat}
\def \tilde{\widetilde}
\def \bar{\overline}

\def\ind{{\mathchoice {\rm 1\mskip-4mu l} {\rm 1\mskip-4mu l}
{\rm 1\mskip-4.5mu l} {\rm 1\mskip-5mu l}}}

\title{\Large\ {\bf Sensitivity analysis for multiscale stochastic reaction networks using hybrid approximations}}

\author{Ankit Gupta and Mustafa Khammash\\
}
\date{\today}
\maketitle
\begin{abstract}
We consider the problem of estimating parameter sensitivities for stochastic models of multiscale reaction networks. These sensitivity values are important for model analysis, and, the methods that currently exist for sensitivity estimation mostly rely on simulations of the stochastic dynamics. This is problematic because these simulations become computationally infeasible for multiscale networks due to reactions firing at several different timescales. However it is often possible to exploit the multiscale property to derive a ``model reduction" and approximate the dynamics as a \emph{Piecewise Deterministic Markov process} (PDMP), which is a hybrid process consisting of both discrete and continuous components. The aim of this paper is to show that such PDMP approximations can be used to accurately and efficiently estimate the parameter sensitivity for the original multiscale stochastic model. We prove the convergence of the original sensitivity to the corresponding PDMP sensitivity, in the limit where the PDMP approximation becomes exact. Moreover we establish a representation of the PDMP parameter sensitivity that separates the contributions of discrete and continuous components in the dynamics, and allows one to efficiently estimate both contributions. 
\end{abstract}



\medskip

\setcounter{equation}{0}

\section{Introduction} \label{sec:intro}

The development of experimental technologies over the past two decades has shed light on the stochastic nature of dynamics inside biological cells \cite{Noise,Elowitz}. Indeed the small volume of a cell and the presence of several biomolecular species with low copy-numbers or population-sizes pushes the intracellular reaction dynamics outside the regime of validity of deterministic descriptions based on ordinary differential equations (ODEs). Consequently, it has been acknowledged that the evolution of the copy-numbers of the species in an intracellular reaction network is better represented by a continuous-time Markov chain (CTMC) model which naturally accounts for randomness in the dynamics, caused by the intermittent nature of the reactions in the presence of low copy-number species \cite{Goutsias}. In the past few years such stochastic models have become a popular tool for studying the effects of dynamical randomness \cite{Rao,McAdams}.

Estimation of the sensitivity of various system observables w.r.t. model parameters is often very important. The values of these parameters are typically inaccurate or uncertain, and parameter sensitivities provide a way to quantify the effects of this imprecision. If the sensitivity value of some output of interest w.r.t. model parameter $\theta$ is large, then the effects of imprecision will also be significant and one might devote more effort in determining $\theta$ more precisely. Apart from this obvious application, sensitivity values are useful for fine-tuning outputs \cite{Feng}, understanding network architectures \cite{Stelling} and parameter inference \cite{Fink2009}.

In deterministic models with dynamics represented as ODEs, estimation of parameter sensitivities is quite straightforward and it can be accomplished by solving another coupled system of ODEs (see Section \ref{sec:mainresults}). This simplicity breaks down for stochastic models and estimation of parameter sensitivity becomes quite challenging. This is mainly due to the intractability of the associated \emph{Chemical Master Equation} (CME) (\ref{defn_cme}) that describes the evolution of the probability distribution of the CTMC representing reaction dynamics. Most of the existing methods for estimating parameter sensitivities in the stochastic setting require simulations of the \emph{exact} stochastic dynamics \cite{IRN,Gir,KSR1,KSR2,DA,Gupta,Gupta2}, with methods such as Gillespie's \emph{stochastic simulation algorithm} (SSA) \cite{GP} or the \emph{next reaction method} (NRM) \cite{NR}. These simulation methods generate each reaction event, making them computationally infeasible for \emph{multiscale} reaction networks that are characterized by reactions firing at several different timescales. For such networks, the simulation proceeds at a speed which is inversely proportional to the \emph{fastest} timescale, and hence generating the whole stochastic trajectory takes a very long time. Recently in \cite{gupta2017estimation}, a new approach for estimating sensitivities has been proposed that only requires stochastic trajectories simulated with approximate tau-leap methods \cite{tleap1,tleap2}, that save on computational costs by aggregating reaction firings over small time-intervals. While this approach might be useful in certain contexts, it is likely to be inadequate for multiscale networks because of the ``stiffness" issue which limits the accuracy of tau-leap schemes \cite{Rathinam2003,cao2004numerical}.

As multiscale networks are commonplace in Systems Biology \cite{Ball,HWKang,crudu2009hybrid}, it is important to develop theoretical and computational tools for sensitivity estimation for stochastic models of such networks. This paper is a contribution in this direction, and it aims to demonstrate how sensitivity analysis can be efficiently carried out using ``model reductions" derived by exploiting multiscale properties of a network \cite{crudu2009hybrid,HWKang}. In such approaches, the complexity is reduced by removing the \emph{fast} reactions in the network by either applying the \emph{quasi-stationary approximation} (QSA) \cite{ssSSA,weinan1,weinan2} or by treating the firing of these reactions continuously (as in deterministic models) rather than discretely. Often both these simplifications need to be applied together to obtain a \emph{fully} reduced model that can accurately capture the original CTMC dynamics for the multiscale model but is computationally much simpler to simulate. A systematic procedure for deriving these reduced models was provided in \cite{HWKang} by explicitly accounting for the fact that disparities in reaction timescales could arise due to two reasons - differences in the magnitudes of reaction-rate constants and variation in copy-number scales of the network species. In this framework, the states of certain species are \emph{discrete} copy-numbers while for other species the states are concentrations that evolve \emph{continuously}. To apply QSA, an ergodic subnetwork that operates at the \emph{fastest} timescale is identified and its stochastic dynamics is assumed to relax to stationarity almost instantaneously at the slower timescale. After application of QSA often any fast reactions that still remain, only affect the remaining species in such a way that they can be treated continuously. As the discrete species are still evolving as a CTMC, the dynamics of the resulting reduced model is a hybrid process with both discrete and continuous components. In particular, this process is a \emph{Piecewise Deterministic Markov process} (PDMP) \cite{davis1993markov} which has received a lot of attention in recent years, especially because of their importance in modeling multiscale biological phenomena \cite{rudnicki2017piecewise}.

The model reduction arguments we just mentioned are described mathematically in Section \ref{sec:pdmpconv}, where it is shown how the original multiscale CTMC dynamics converges to the PDMP limit, as a certain scaling parameter $N$, denoting system volume or the overall population-size, tends to infinity. In this context, the main contributions of the paper are twofold. Firstly we establish that under certain conditions, as the scaling parameter $N  \to \infty$, the sensitivities for the original multiscale model are guaranteed to converge to the corresponding sensitivities for the limiting PDMP model (Theorem \ref{thm:convergence}) and secondly we express the PDMP sensitivity in a simple form which clearly separates the sensitivity contribution due to discrete and continuous reactions (Theorem \ref{thm:representation}). As we demonstrate, such a representation of PDMP sensitivity makes it amenable to estimation by existing methods, and helps in leveraging that the computation of parameter sensitivity is quite straightforward and very efficient for continuous dynamics. Observe that despite the PDMP convergence of the multiscale dynamics, the sensitivity convergence result, Theorem \ref{thm:convergence}, in non-trivial because the $\theta$-derivative and the limit $N  \to \infty$ may not necessarily commute. This commutation needs to be checked by careful analysis and the main difficulty arises due to interactions between discrete and continuous part of the dynamics, through the propensity functions. These interactions create intricate dependencies which need to be carefully disentangled to prove the aforementioned results. This is done using coupling arguments based on the random time-change representation by Kurtz (see Chapter 6 in \cite{EK}). Note that the sensitivity convergence result for QSA application, analogous to Theorem \ref{thm:convergence}, was established in our earlier paper \cite{gupta2014sensitivity}. Hence our result on sensitivity convergence in the PDMP limit is complementary to the result in \cite{gupta2014sensitivity}, and it completes the mathematical analysis needed to show how accurate sensitivity approximations can be derived using model reductions for multiscale networks. We illustrate this using a couple of computational examples in Section \ref{sec:example}.

This paper is organized as follows. In Section \ref{sec:prelim} we describe stochastic models for multiscale reaction networks and present the associated PDMP convergence result. In Section \ref{sec:mainresults} we state and explain our main results. These results are illustrated with examples in Section \ref{sec:example} and proved in the Appendix. Finally in Section \ref{sec:concl} we conclude and present directions for future research.

\section*{Notation}

Throughout the paper $\N_0$, $\Z$, $\R$ and $\R_+$ denote the set of nonnegative integers, all integers, all real numbers and nonnegative real numbers respectively. The cardinality of any set $A$ is denoted by $|A|$ and for any two real numbers $a,b \in \R$ their minimum is denoted by $a \wedge b$. In any Euclidean space $\R^n$, $\langle \cdot, \cdot \rangle$ denotes the standard inner product and $\| \cdot \|$ is the corresponding norm. The vector of all zeros and all ones in $\R^n$ is denoted by ${\bf 0}$ and ${\bf 1}$ respectively. The transpose of a vector/matrix $M$ is $M^*$. All vectors should be treated as column vectors unless otherwise stated.

\section{Preliminaries} \label{sec:prelim}

\subsection{A stochastic reaction network}

Consider a reaction network with $S$ species, denoted by ${\bf S}_1,\dots,{\bf S}_{S}$, which can undergo $K$ reactions of the form
\begin{align}
\label{reactionform}
\sum_{i=1}^{S} \nu_{ik} {\bf S}_i \longrightarrow  \sum_{i=1}^{S} \nu'_{ik} {\bf S}_i. 
\end{align}
Here $\nu_{ik}$ and $\nu'_{ik}$ are nonnegative integers that correspond to the number of molecules of species ${\bf S}_i$ that the $k$-th reaction \emph{consumes} and \emph{produces} respectively. In the stochastic model of such a reaction network, the state at any time is the vector \\ $x=(x_1,\dots,x_S) \in \N^S_0$ of copy-numbers or population-sizes of all the $S$ species and these states evolve according to a continuous time Markov chain (CTMC). For each $k =1,\dots,K$, let $\zeta_k = ( \nu'_{1k} - \nu_{1k} , \dots,  \nu'_{S k} - \nu_{Sk}  ) \in \Z^S$ be the vector of copy-number changes caused by reaction $k$ and let $\lambda'_k(x)$ be the rate of firing of reaction $k$ at state $x$. The function $\lambda'_k : \N^{S}_0  \to \R_+$ is called the \emph{propensity function} and the vector $\zeta_k$ is called the \emph{stoichiometric} vector for reaction $k$. Noting that the state will change from $x$ to $(x + \zeta_k)$ upon firing of reaction $k$, we can specify the CTMC describing the reaction kinetics by its generator\footnote{The generator of a Markov process is an operator specifying the infinitesimal rate of change of the distribution of the process (see Chapter 4 in \cite{EK} for more details)}
\begin{align*}
\mathbb{A} f (x) = \sum_{k = 1}^K \lambda'_k(x) ( f(x+\zeta_k) - f(x) ),
\end{align*}
defined  for any bounded function $f$ on $\N^S_0$. An alternative and often a more useful way to represent this CTMC $( X(t) )_{t \geq 0}$ is by Kurtz's random time-change representation (see Chapter 7 in \cite{EK}) given by
\begin{align*}
X(t) = X(0) + \sum_{k=1}^K Y_k\left(  \int_{0}^{t} \lambda' _k(  X(s) )  ds \right) \zeta_k,
\end{align*}
where $\{ Y_k : k=1,\dots,K \}$ is a family of independent unit rate Poisson processes.

The \emph{forward} Kolmogorov equation for this CTMC is popularly known as the \emph{Chemical Master Equation} (CME) and it describes the time-evolution of the probability distribution of $ (X(t))_{t  \geq 0}$. The CME is a system of following ordinary differential equations (ODEs) 
\begin{align}
\label{defn_cme}
\frac{d p(x,t)}{dt} = \sum_{k=1}^K \lambda'_k ( x -\zeta_k ) p( x - \zeta_k,t) - \sum_{k=1}^K \lambda'_k(x) p(x,t)
\end{align}
for each accessible state $x \in \N^S_0$. Here $p(x,t) = \mathbb{P}( X(t) = x )$ is the probability that the reaction dynamics $ (X(t))_{t \geq 0 } $ is in state $x$ at time $t$. Often the number of accessible states is very large or infinite and so the CME cannot be directly solved. Instead its solutions are estimated via Monte Carlo approaches like Gillespie's well-known \emph{stochastic simulation algorithm} (SSA) \cite{GP} or the  \emph{next reaction method} (NRM) \cite{NR,anderson2007modified} which is based on the random time-change representation \eqref{defn_cme}. Such methods recover the correct solution of the CME in the infinite sample-size limit, but they simulate each sample path of $(X(t) )_{t \geq 0}$ \emph{exactly}, accounting for each and every reaction event. Unsurprisingly these approaches become exorbitantly computationally expensive for even moderately sized reaction networks. To mitigate this problem, alternate simulation approaches, like $\tau$-leap methods \cite{tleap1,tleap2} have been developed, that sacrifice the exactness of the sample paths for gain in computational feasibility.

\subsection{Multiscale models} \label{sec:multiscalemodels}

Many reaction networks encountered in fields like Systems Biology are multiscale in nature, i.e. they have reactions firing at several timescales. The exact simulation methods like SSA or NRM often become completely impractical for multiscale networks, because these methods proceed with a time-step which is inversely proportional to the fastest reaction timescale. This issue is related to the problem of ``stiffness" of the CME corresponding to multiscale networks, which is known to cause difficulties in using tau-leap simulation approaches \cite{Rathinam2003,cao2004numerical}.

As working with a \emph{pure-jump} CTMC description of the dynamics is difficult for multiscale networks, researchers have come up with results that show that under certain conditions it is possible to exploit the multiscale features of a network, in order to ``reduce" the model complexity and work with simplified descriptions of the dynamics, that are good approximations of the original dynamics. Often this simplified process is \emph{hybrid}, having both discrete and continuously evolving components, that are mutually coupled \cite{HWKang,crudu2009hybrid}. A commonly arising hybrid process is a \emph{Piecewise Deterministic Markov Process} (PDMP), whose continuous part evolves as a system of ODEs while the discrete part undergoes jumps like the original CTMC. 

A systematic, mathematically rigorous approach for identifying PDMP approximations of multiscale stochastic reaction networks is developed in \cite{HWKang} under the assumption of \emph{mass-action kinetics} \cite{DASurvey}, i.e. each propensity function $\lambda'_k : \N^S_0 \to \R_+$ is given by
\begin{align}
\label{defn:massactionkinetics}
\lambda'_k (x_1,\dots,x_S) =  \kappa'_k  \prod_{i =1}^{S} \frac{ x_i(x_i-1)\dots (x_i - \nu_{ik} +1 ) }{  \nu_{ik} ! }
\end{align}
where $\kappa'_k$ is the reaction-rate constant for the $k$-th reaction and $\nu_{ik}$-s are as in \eqref{reactionform}. In the rest of this section, we describe this approach and state the limit theorem that validates the PDMP approximation of the reaction dynamics of a multiscale network. The multiplicity in reaction timescales could be due to differences in species copy-number scales and to discrepancies in the magnitudes of reaction rate-constants. Hence one needs to account for both these sources of variations to construct the required dynamical approximation. Moreover whether a reaction timescale is \emph{fast} or \emph{slow} can only be decided relative to the timescale of observation which must be chosen a priori. To accommodate these considerations a scaling parameter $N_0$ is chosen that corresponds to the system volume or the magnitude of the population-size of an abundant species. Thereafter each species ${\bf S}_i$ is assigned an \emph{abundance} factor $\alpha_i \geq 0$ and each reaction $k$ is assigned a scaling constant $\beta_k \in \R$. These parameters serve as \emph{normalizing} constants, in the sense that if $X_i(t)$ denotes the copy-number of species ${\bf S}_i$ at time $t$, then $N^{ -\alpha_i }_0  X_i( N^{\gamma}_0 t)$ is roughly of order $1$ or $O(1)$ on the timescale of interest $\gamma \in \R$. Similarly the scaled rate constant $\kappa_k = \kappa'_k N^{  -\beta_k}_0$ is $O(1)$ for each reaction $k$. There exists computational approaches to automatically select these normalizing parameters $\alpha_i$-s and $\beta_k$-s \cite{hepp2015adaptive}.

Let $(X^{N_0} (t) )_{ t \geq 0 } $ be the $\N^{S}_0$-valued CTMC representing the dynamics of the multiscale network. Once the choice of parameters $\alpha_i$-s and $\beta_k$-s has been made, we define our scaled process $( Z^{N_0,\gamma}(t)  )_{t \geq 0}$ as
\begin{align}
\label{defn:scaledprocess}
Z^{N_0,\gamma}(t) = \Lambda_{N_0} X^{ N_0 }( N^{\gamma}_0 t),
\end{align}
where $ \Lambda_{N_0} = \textnormal{Diag}( N_0^{ - \alpha_1} , \dots, N_0^{ - \alpha_{S} } )$ is the $S \times S$ diagonal matrix containing all the species scaling factors. Our aim is to analyze the original process $X^{N_0}$ at the observation timescale $\gamma$, using the scaled process $Z^{N_0,\gamma}$. Note that for each species ${\bf S}_i$, $\alpha_i$ is chosen in such a way that the \emph{normalized abundance} of this species, given by $Z^{N_0,\gamma}_i(t) =  N^{ -\alpha_i }_0  X^{ N_0 }( N^{\gamma}_0 t)$ remains $O(1)$ over compact time-intervals.

Replacing $N_0$ by $N$ we arrive at a family of processes $\{Z^{N,\gamma}\}$ parameterized by $N$. If we can show that $Z^{N,\gamma}$ converges in the sense of distributions in the Skorohod topology on $\R^S$(see Chapter 3 in \cite{EK}) to some process $Z$, then it would imply that for a large value of $N_0$ we have
\begin{align}
\label{defn:distributionalconv}
\E( f(  Z^{N_0,\gamma}(t)  ) ) \approx \lim_{N  \to \infty} \E( f(  Z^{N,\gamma}(t)  ) ) = \E( f( Z(t) ) ),
\end{align}
 for any $t \geq 0$ and any continuous and bounded function $f : \R^S \to \R$. In other words, the probability distribution of $Z(t)$ is close to the probability distribution of $ Z^{N_0,\gamma}(t) =  \Lambda_{N_0} X^{ N_0 }( N^{\gamma}_0 t)$. In this paper we are interested in situations where the limiting process is a PDMP that is usually much simpler to simulate and analyze than the original multiscale process. Hence passing to the limit achieves a ``model reduction" which can be significantly helpful in analyzing multiscale networks. We now discuss this convergence to a PDMP in greater detail.

\subsection{PDMP Convergence}\label{sec:pdmpconv}

The random time-change representation of process $Z^{N,\gamma}$ is
\begin{align}
\label{scaledprocess:rtc1}
Z^{N,\gamma} (t) = Z^{N,\gamma} (0) + \sum_{k=1}^K Y_k \left( N^\gamma  \int_0^{t  }   \lambda'_k (  \Lambda^{-1}_{N}  Z^{N,\gamma} (s) )ds \right) \Lambda_N \zeta_k,  
\end{align}
where $\{ Y_k : k=1,\dots,K \}$ is a family of independent unit rate Poisson processes as before and $Z^{N,\gamma} (0) $ is the initial condition chosen as
\begin{align*}
Z^{N,\gamma} (0) = Z^{N_0,\gamma} (0) =   \Lambda_{N_0} X^{N_0} (0).
\end{align*}
As $\lambda'_k$ is given by the mass-action form \eqref{defn:massactionkinetics} with $\kappa_k = \kappa'_k N^{ \beta_k } $ it can be seen that for any $z$
\begin{align}
\label{propapprox}
\lambda'_k (  \Lambda^{-1}_{N}  z ) =N^{ \beta_k + \langle \nu_k , \alpha \rangle } \lambda^{N}_k (z),
\end{align}
where $\nu_k = ( \nu_{1k} ,\dots, \nu_{S k} )$, $\alpha = ( \alpha_1,\dots, \alpha_S )$ and $\langle \cdot, \cdot \rangle$ denotes the standard inner product on $\R^S$. The function $ \lambda^{N}_k $ is given by
\begin{align*}
 \lambda^{N}_k (z_1,\dots,z_S) = \kappa_k \prod_{i=1}^S \frac{z_i(z_i - N^{ -\alpha_i}  )  \dots  (z_i - N^{ -\alpha_i} \nu_{ik} +  N^{ -\alpha_i}  )  }{ \nu_{ik} ! }
\end{align*}
and it is immediate that 
\begin{align}
\label{defn:prop_limit}
\lim_{ N  \to \infty }   \lambda^{N}_k (z_1,\dots,z_S) & =  \lambda_k (z_1,\dots,z_S) \notag \\
&:=  \kappa_k \left( \prod_{ i \in \mathcal{S}_d } \frac{z_i(z_i - 1 )  \dots  (z_i  -  \nu_{ik} + 1 )  }{ \nu_{ik} ! } \right) \left( \prod_{ i \in \mathcal{S}_c } \frac{z^{ \nu_{ik} }_i }{ \nu_{ik} ! } \right).
\end{align}
Here $\mathcal{S}_d := \{ i =1,\dots,S : \alpha_i = 0 \}$ and $\mathcal{S}_c :=  \{ i =1,\dots,S : \alpha_i > 0 \}$ are sets of species with zero abundance factors and positive abundance factors respectively. Species in $\mathcal{S}_d$ will be treated \emph{discretely} while the species in $\mathcal{S}_c$ will be treated \emph{continuously} by the limiting PDMP described by process $Z$. Let $S_c = | \mathcal{S}_c  |$ and $S_d = | \mathcal{S}_d  |$. Without loss of generality we can assume that $\mathcal{S}_c = \{1,\dots, S_c\}$ and $\mathcal{S}_d = \{S_c+ 1,\dots, S\}$. One can view $\rho_k :=\beta_k +  \langle \alpha,\nu_k \rangle $ as the ``natural" timescale for reaction $k$, which takes into account both reactant copy-number variation and the magnitude of the associated rate constant. Due to \eqref{propapprox}, the random time-change representation \eqref{scaledprocess:rtc1} for species ${\bf S}_i$ has the form
\begin{align}
\label{scaledprocess:rtc2}
Z^{N,\gamma}_i(t) =  Z^{N,\gamma}_i(0) + \sum_{k=1}^K  N^{ -\alpha_i} Y_k \left( N^{\rho_k+ \gamma}  \int_0^{t  }   \lambda^N_k (  Z^{N,\gamma} (s) )ds \right) \zeta_{ki}.
\end{align}
Hence the timescale at which the normalized mass of species ${\bf S}_i$ changes at a rate of $O(1)$ is $\gamma_i := \alpha_i - \max\{  \rho_k : \zeta_{ik} \neq 0 \}$. Higher $\gamma_i$ implies a slower rate of change for species ${\bf S}_i$ and so the \emph{fastest} timescale for the evolution of the normalized population-sizes of the species is given by 
\begin{align}
\label{pdmp_conv_timescale}
r = \min_{i =1,\dots,S}  \gamma_i. 
\end{align} 
We now set the observation timescale as $\gamma = r$ and study the limiting behavior of process $Z^{N,r}$ as $N  \to \infty$. Note that for any reaction $k$ and species ${\bf S}_i$ such that $\zeta_{ki} \neq 0$ we must have
\begin{align*}
r + \rho_k \leq \gamma_i + \rho_k \leq \gamma_i + \max\{  \rho_k : \zeta_{ik} \neq 0 \} \leq \alpha_i.
\end{align*}
If this inequality is strict then reaction $k$ will not affect species ${\bf S}_i$ in the limit $N \to \infty$, because the rate at which this reaction modifies the copy-numbers of species ${\bf S_i}$ is $(r + \rho_k)$ which is strictly less than the abundance factor $\alpha_i$ for this species. So for each $k$, we define $\hat{\zeta}_k$ to be the vector obtained by transforming the stoichiometric vector $\zeta_k$ as
\begin{align*}
\hat{\zeta}_{ik} = \left\{
\begin{array}{cc}
\zeta_{ik} & \textnormal{ if } \alpha_i = (r + \rho_k) \\
0 & \textnormal{ otherwise}.
\end{array} \right.
\end{align*}

Let $\mathcal{R}_c:= \{ k =1,\dots,K : r + \rho_k > 0   \} $ and $\mathcal{R}_d = \{ k = 1,\dots,K: r + \rho_k = 0   \}$ denote the sets of \emph{continuous} and \emph{discrete} reactions respectively and suppose that process $Z^{N,r}$ converges to some process $Z$ as $N  \to \infty$. The \emph{Poisson law of large numbers} (PLLN) ensures that for any $T > 0$
\begin{align*}
\sup_{0 \leq t \leq T } \left| \frac{ Y(Nt) }{N} -t \right|   \stackrel{a.s.}{\longrightarrow} 0 \ \textnormal{ as } \ N \to \infty.
\end{align*}
This result along with limit \eqref{defn:prop_limit}, suggests that for any reaction $k \in \mathcal{R}_c$
\begin{align*}
N^{ -\alpha_i} Y_k \left( N^{\rho_k+ \gamma}  \int_0^{t  }   \lambda^N_k (  Z^{N,r} (s) )ds \right) \zeta_{ki} \stackrel{a.s.}{\longrightarrow}  \left( \int_0^{t  } \lambda_k (  Z(s) )ds\right)  \hat{\zeta}_{ki}  \ \textnormal{ as } \ n \to \infty.
\end{align*}
Hence in the limit, the discrete and intermittent nature of the firing of reaction $k$ is lost and replaced by a continuous process. On the other hand for any $k \in \mathcal{R}_d$, this discreteness is preserved and we have
\begin{align*}
N^{ -\alpha_i} Y_k \left( N^{\rho_k+ \gamma}  \int_0^{t  }   \lambda^N_k (  Z^{N,r} (s) )ds \right) \zeta_{ki} \stackrel{a.s.}{\longrightarrow} Y_k \left( \int_0^{t  } \lambda_k (  Z(s) )ds\right)  \hat{\zeta}_{ki}  \ \textnormal{ as } \ n \to \infty.
\end{align*}
From these limits one can argue that the limiting process $Z$ is a PDMP and it can be expressed as
\begin{align*}
Z(t) = (x(t) , U(t) )
\end{align*}
where $x(t) \in \R^{S_c}$ and $U(t) \in \N^{S_d}_0$ are time $t$ state-vectors for the continuous species in $\mathcal{S}_c$ and the discrete species in $\mathcal{S}_d$ respectively. This PDMP evolves according to 
\begin{align}
\label{defn_pdmp}
x(t) &= x(0) + \sum_{k \in \mathcal{R}_c} \left( \int_{0}^t \lambda_k( x(s), U(s) ) ds \right) \zeta^{(c)}_k \\
U(t) & = U(0) + \sum_{k \in \mathcal{R}_d} Y_k \left( \int_{0}^t \lambda_k( x(s), U(s) ) ds  \right) \zeta^{(d)}_k,  \notag
\end{align}
with $\{Y_k\}$ being independent unit-rate Poisson processes as above, and $\zeta^{(c)}_k$ (resp. $ \zeta^{(d)}_k$) denoting the first $S_c$ (resp. last $S_d$) components of the vector $\hat{\zeta}_k$. The convergence of the process $Z^{N,r}$ to PDMP $Z$ is established rigorously by Theorem 4.1 in \cite{HWKang}. This convergence will hold until the \emph{explosion} time of process $Z$ but in this paper we assume that this explosion time is infinite almost surely (see Assumption \ref{assmp:key}).

Even though the above framework for deriving PDMP convergence is only described for mass-action kinetics, even with more general kinetics the same PDMP convergence result will apply as long as we can make sure that \eqref{propapprox} holds and $\lambda^N_k$ converges point-wise to some limiting function $\lambda_k$. We end this section with an important remark regarding model reduction for multiscale networks.
\begin{remark}
\label{rem:qssa}
As discussed in Section \ref{sec:intro}, often the limiting PDMP model is derived after some of the fast reactions among discrete species have been eliminated using the quasi-stationary assumption (QSA). This has been described in detail in \cite{HWKang} and in our earlier paper \cite{gupta2014sensitivity} we showed that under some conditions the parameter sensitivities converge under QSA application, which can be used to efficiently estimate parameter sensitivity for multiscale networks. In this paper we derive an analogous result, Theorem \ref{thm:convergence}, for PDMP limits and thereby complete the story on sensitivity analysis of multiscale networks using model reductions.   
\end{remark}

\subsection{PDMP Simulation}

In this section we briefly discuss how the PDMP $(Z(t))_{t \geq 0}$ defined by \eqref{defn_pdmp} can be simulated. Many strategies exist for efficient simulation of PDMPs \cite{crudu2009hybrid,duncan2016hybrid,hepp2015adaptive}. In this paper we shall simulate PDMPs by adapting Algorithm 2 in \cite{duncan2016hybrid}, which is a generalization of the \emph{next reaction method} (NRM) \cite{NR} and is based on representation \eqref{defn_pdmp} using time-changed Poisson processes. We present this simulation scheme as Algorithm \ref{algo_pdmp_sim}. In this algorithm, the ODEs for the evolution of the continuous states (i.e. $x(t)$) are solved with Euler discretization and the \emph{internal times} 
\begin{align*}
T_k = \int_{0}^t \lambda_k( x(s), U(s) ) ds
\end{align*}
for all the discrete reactions $k \in \mathcal{R}_d$ are updated in the same way. For each such reaction, $P_k$ denotes the next jump time of the unit-rate Poisson process $Y_k$ and as soon as $T_k$ exceeds $P_k$, reaction $k$ is fired, the discrete state $U(t)$ is updated and $P_k$ is assigned a new value according to a realization of an independent exponential random variation with rate $1$.

\begin{algorithm}[H]  
\caption{Simulates the PDMP $(x(t),U(t))_{t \geq 0}$ in the time-interval $[0,T]$ with initial state $(x_0,U_0)$.}      
 \label{algo_pdmp_sim}
 \begin{algorithmic}[1]
\State {\bf Initialization:} Set $t = 0$, $x(t) = x_0$ and $U(t) = U_0$. For each $k \in \mathcal{R}_d$ set $T_k = 0$ and set $P_k = -\log(u)$ for $u \sim \textnormal{Uniform}[0,1]$.
\State {\bf Continuous update:} Pick a time-discretization step $\delta t$. Set
\begin{align*}
x(t + \delta t) &= x(t) +  \left( \sum_{k \in \mathcal{R}_c} \lambda_k( x(t), U(t) )  \zeta^{(c)}_k \right) \delta t \\
\textnormal{and} \qquad T_k & = T_k + \lambda_k( x(t) ,U(t) ) \delta t \qquad \textnormal{for each} \qquad k \in \mathcal{R}_d.
\end{align*}
\State {\bf Discrete update:} For each reaction $k \in \mathcal{R}_d$, if $T_k > P_k$ then set
\begin{align*}
U(t + \delta t) = U(t) + \zeta^{(d)}_k \qquad \textnormal{and} \qquad P_k = P_k -\log(u) \quad \textnormal{for} \quad u \sim \textnormal{Uniform}[0,1].
\end{align*}
\State Update $t \gets t+\delta t$. If $t > T$ then {\bf stop}, or else return to step 2.
\end{algorithmic}
\end{algorithm}

\section{Main Results} \label{sec:mainresults}

We now assume that propensity functions of our multiscale network depend on a real-valued parameter $\theta$. We denote these propensity functions as $\lambda'_k(x, \theta)$ and we let $( X^{N_0}_\theta(t) )_{t \geq 0}$ be the CTMC describing the dynamics of the multiscale network. Let $( Z^{N_0}_\theta(t) )_{t \geq 0 }$ be the scaled process defined by \eqref{defn:scaledprocess} with the timescale of observation chosen as $\gamma = r$ (see \eqref{pdmp_conv_timescale}). We assume that convergence to a PDMP $(Z_\theta(t))_{t \geq 0 }$ holds, where $Z_\theta(t) = (x_\theta(t) , U_\theta(t) )$ evolves according to
\begin{align}
\label{defn_pdmp_param}
x_\theta(t) &= x_0 + \sum_{k \in \mathcal{R}_c} \left( \int_{0}^t \lambda_k( x_\theta(s), U_\theta(s) ,\theta) ds \right) \zeta^{(c)}_k \\
U_\theta(t) & = U_0 + \sum_{k \in \mathcal{R}_d} Y_k \left( \int_{0}^t \lambda_k( x_\theta(s), U_\theta(s),\theta ) ds  \right) \zeta^{(d)}_k,  \notag
\end{align}
with $\{Y_k\}$ being independent unit-rate Poisson processes, and $\lambda_k( \cdot,\cdot ,\theta)$ being $\theta$-dependent propensity functions that are related to the original propensity functions $\lambda'_k(\cdot, \theta)$ in the same way as in Section \ref{sec:pdmpconv}. The initial state $(x_0,U_0)$ of this PDMP is deterministic and independent of $\theta$.

Our goal is to estimate the following $\theta$-sensitivity value for the original multiscale network
\begin{align}
\label{defn:sens_main}
S^{N_0}_\theta(f,T) = \frac{ \partial }{ \partial \theta} \E( f(  \Lambda_{N_0}  X^{N_0}_\theta( N^{r}_0  T ) ) ) =  \frac{ \partial }{ \partial \theta} \E( f( Z^{N_0}_\theta(  T ) ) ),
\end{align}
where the function $f : \R^{S_c} \times \N^{S_d}_0 \to \R$ captures the output of interest and $T$ is the time of observation. As discussed in Section \ref{sec:intro}, estimating this sensitivity directly via existing simulation-based approaches \cite{IRN,KSR1,KSR2,DA,Gupta,Gupta2,Gir} is highly cumbersome due to the multiscale nature of the network. Therefore we would like to exploit the PDMP convergence $Z^N_\theta \Rightarrow Z_\theta$ to obtain an accurate estimate of $S^{N_0}_\theta(f,T) $. For this purpose, our first result, Theorem \ref{thm:convergence}, proves that we have 
\begin{align}
\label{sens_conv_result}
\lim_{N  \to \infty}  \frac{ \partial }{ \partial \theta} \E( f( Z^{N}_\theta(  T ) ) ) = \hat{S}_\theta(f,T):=  \frac{ \partial }{ \partial \theta} \E( f( Z_\theta(  T ) ) ).
\end{align}
Once this convergence is verified, then for large $N_0$ it is possible to estimate $S^{N_0}_\theta(f,T) $ by estimating $\hat{S}_\theta(f,T)$ instead. Our second result, Theorem \ref{thm:representation}, will enable us to efficiently estimate the sensitivity  $\hat{S}_\theta(f,T)$ with simulations of the PDMP $Z_\theta$. As performing PDMP simulations is usually much easier than simulating the original multiscale CTMC, our approach can significantly reduce the computational time required to estimate the quantity of interest $S^{N_0}_\theta(f,T) $. This is demonstrated by numerical examples in Section \ref{sec:example}.

Let $F(x,U,\theta)$ be any real-valued function on $\R^{S_c} \times \N^{S_d}_0 \times \R$ which is differentiable in the first $S_c$ coordinates. We denote the gradient of this function w.r.t. these coordinates as $\nabla F(x,U,\theta)$ and the partial derivative w.r.t.\ $\theta$ as $\partial_{\theta} F(x,U ,\theta)$. Moreover for any discrete reaction $k \in \mathcal{R}_d$, $\Delta_{k} F(x,U,\theta)$ refers to the difference
\begin{align*}
\Delta_{k} F(x,U,\theta)= F(x , U  +\zeta^{ (d) }_k ,\theta ) - F(x,U,\theta).
\end{align*}
We now specify certain assumptions that are needed to prove our main results.
\begin{assumption}
\label{assmp:key}
\begin{itemize}
\item[(A)] The observation function $f : \R^{S_c} \times \N^{S_d}_0 \to \R$ is continuously differentiable in the first $S_c$ coordinates.
\item[(B)] For each reaction $k$, as $N  \to \infty$, the convergence $\lambda^N_k(\cdot, \theta) \to  \lambda_k(\cdot, \theta)$ (see \eqref{defn:prop_limit}) holds uniformly over compact sets. The same is true for the derivatives \\ $\nabla \lambda^N_k(\cdot, \theta) \to  \nabla \lambda_k(\cdot, \theta)$ and $\partial_\theta \lambda^N_k(\cdot, \theta) \to \partial_\theta  \lambda_k(\cdot, \theta)$.
\item[(C)] There exists a compact set $C \subset \R_+^{S_c}  \times \N^{S_d}_0$ which is large enough to contain the scaled dynamics in the time interval $[0,T]$ almost surely in the limit $N  \to \infty$, i.e.
\begin{align*}
\limsup_{N \to \infty} \sup_{ t \in [0,T] }\mathbb{P}\left(  Z^N_\theta(t) \notin C  \right) = 0.
\end{align*}
\end{itemize}
\end{assumption}
The conditions in parts (A) and (B) of this assumption are quite mild and likely to be satisfied in most situations of interest. However the condition in part (C) is relatively stricter and it essentially requires the multiscale dynamics to stay inside a compact set, which would only hold if there is a global mass conservation relation among the species. We make this strict assumption here for technical convenience but we note that this condition can be substituted with a weaker condition that imposes restrictions on the growth rate of propensity functions for mass-producing reactions (see \cite{Gupta} for more details).

\begin{theorem}[Sensitivity Convergence]
\label{thm:convergence}
Suppose that Assumption \ref{assmp:key} is satisfied and $Z^N_\theta \Rightarrow Z_\theta$ as $N  \to \infty$. Then convergence \eqref{sens_conv_result} holds, which shows that $\theta$-sensitivity for the multiscale network converges to the corresponding $\theta$-sensitivity for the limiting PDMP model.
\end{theorem}
\begin{proof}
The proof is provided in the Appendix.
\end{proof}

Define $\Psi_{t}(x,U,\theta)$ by
\begin{align}
\label{defn_psi}
\Psi_{t}(x,U,\theta) = \E( f(  x_\theta(t) ,U_\theta(t) )  )
\end{align}
for any $t \geq 0$, where $(x_\theta(t), U_\theta(t))_{ t \geq 0 }$ is the PDMP \eqref{defn_pdmp_param} with initial state \\$(x_\theta(0) ,U_\theta(0) ) = (x,U) \in \R_+^{S_c} \times \N^{S_d}_0$. Also let $y_\theta(t)$ be the partial derivative of $x_\theta(t)$ w.r.t.\ $\theta$, under the restriction that $U_\theta(t)$ remains fixed i.e.\ $y_\theta(t)$ satisfies the Initial Value Problem (IVP)
\begin{align}
\label{defn_y_thetat}
\frac{d y_\theta}{dt} &= \sum_{k \in \mathcal{R}_c} \left(  \partial_{\theta}   \lambda_k (x_\theta(t) , U_\theta(t) ,\theta )  +  \left\langle  \nabla  \lambda_k(x_\theta(t) , U_\theta(t) ,\theta ) , y_\theta(t)  \right\rangle \right) \zeta^{(c)}_k  \\
\textnormal{and}  & \qquad y_\theta(0) = {\bf 0}.  \notag
\end{align}
Henceforth for any reaction $k$ let
\begin{align}
\label{defn_part_tod_der}
D_\theta  \lambda_k (x_\theta(t) , U_\theta(t) ,\theta ) =  \partial_{\theta}   \lambda_k (x_\theta(t) , U_\theta(t) ,\theta )  +  \left\langle  \nabla  \lambda_k(x_\theta(t) , U_\theta(t) ,\theta ) , y_\theta(t)  \right\rangle
\end{align}
denote the $\theta$-derivative of the propensity function $ \lambda_k (x_\theta(t) , U_\theta(t) ,\theta )$, where we \emph{include} the $\theta$-dependence of the continuous state $x_\theta(t)$ but \emph{disregard} the $\theta$-dependence of the discrete state $U_\theta(t)$. We are now ready to state our next result which provides a nice representation for the PDMP sensitivity $\hat{S}_\theta(f,T)$ (see \eqref{sens_conv_result}), that would allow us to efficiently estimate this quantity with PDMP simulations.

\begin{theorem}[Sensitivity Representation]
\label{thm:representation}
Suppose Assumption \ref{assmp:key} holds and let \\$(  Z_\theta(t) )_{ t \geq 0 } = ( x_\theta(t), U_\theta(t) )_{t \geq 0}$ be the PDMP given by \eqref{defn_pdmp_param}. Also let $y_\theta(t)$ be the solution of IVP \eqref{defn_y_thetat} and let $D_\theta \lambda_k(  x_\theta(t), U_\theta(t) ,\theta)$ be defined by \eqref{defn_part_tod_der}. Then we can express sensitivity $\hat{S}_\theta(f,T)$ as the sum
\begin{align}
\label{sens_pdmprep1}
\hat{S}_\theta(f,T) &= \hat{S}^{(c)}_\theta(f,T) +  \hat{S}^{(d)}_\theta(f,T),
\end{align}
where 
\begin{align*}
 \hat{S}^{(c)}_\theta(f,T)  = \E\left[  \left \langle \nabla f (x_\theta(T) ,U_\theta(T) ) , y_\theta(T)   \right\rangle \right]
\end{align*}
is the sensitivity contribution of the continuous part of the dynamics and 
\begin{align*}
 \hat{S}^{(d)}_\theta(f,T)  = \sum_{k \in \mathcal{R}_d } \E \left[ \int_0^T   D_{\theta}   \lambda_k (x_\theta(t) , U_\theta(t) ,\theta )    \Delta_k \Psi_{T- t}( x_\theta(t), U_\theta(t), \theta ) dt  \right]
\end{align*}
is the sensitivity contribution of the discrete part of the dynamics.
\end{theorem}
\begin{proof}
The proof is provided in the Appendix.
\end{proof}

Note that if there are no discrete reactions (i.e. $\mathcal{R}_d = \emptyset$) then $U_\theta(t)$ remains fixed to its initial state $U_0$, the PDMP $( Z_\theta(t) )_{t \geq 0 }$ essentially reduces to the deterministic process $( x_\theta(t) )_{t  \geq 0}$ and the $\theta$-derivative of the state $x_\theta(t)$ is simply $y_\theta(t)$ obtained by solving IVP \eqref{defn_y_thetat}. Consequently $ \hat{S}^{(d)}_\theta(f,T) =0$ and
\begin{align*}
\hat{S}_\theta(f,T) = \frac{\partial}{ \partial \theta}  f( x_\theta(T)  , U_0 ) =  \left \langle  \nabla f( x_\theta(T), U_0 ) , y_\theta(T)  \right \rangle =    \hat{S}^{(c)}_\theta(f,T),
\end{align*}
which is well-known for deterministic system. On the other extreme suppose that there are no continuous reactions (i.e. $\mathcal{R}_c = \emptyset$). In this case, $x_\theta(t) = x_\theta(0) = x_0$  and $y_\theta(t) = y_\theta(0) = {\bf 0}$ for all $t \geq 0$ and the PDMP $( Z_\theta(t) )_{t \geq 0 }$ essentially reduces to a CTMC $( U_\theta(t) )_{t  \geq 0}$. Hence $ \hat{S}^{(c)}_\theta(f,T) = 0$ and Theorem 3.1 in \cite{gupta2017estimation} shows that
\begin{align*}
\hat{S}_\theta(f,T) = \sum_{k \in \mathcal{R}_d } \E \left[ \int_0^T   \partial_{\theta}   \lambda_k (x_0 , U_\theta(t) ,\theta )    \Delta_k \Psi_{T- t}( x_0, U_\theta(t), \theta ) dt  \right] =  \hat{S}^{(d)}_\theta(f,T),
\end{align*}
which is consistent with Theorem \ref{thm:representation}. Using this representation, an unbiased estimator for $\hat{S}^{(d)}_\theta(f,T)$ can be constructed using simulations of the CTMC $( U_\theta(t) )_{t  \geq 0}$. This was done with the \emph{Integral Path Algorithm} (IPA) in \cite{gupta2017estimation} and it is straightforward to generalize IPA to provide an estimator for $\hat{S}^{(d)}_\theta(f,T)$ for a PDMP, even in the presence of continuous reactions. We shall employ this method to estimate $\hat{S}^{(d)}_\theta(f,T)$ in the numerical example considered in Section \ref{sec:example1}. However it is important to point out that any of the existing sensitivity estimation methods for reaction network described by CTMCs \cite{IRN,Gir,KSR1,KSR2,DA,Gupta,Gupta2} can be used to estimate $\hat{S}^{(d)}_\theta(f,T)$ by adapting these methods to handle time-varying propensities. To see this note that Theorem 3.1 in \cite{gupta2017estimation} allows us to express $\hat{S}^{(d)}_\theta(f,T)$ as
\begin{align}
\label{result_equiv}
\hat{S}^{(d)}_\theta(f,T) =  \left. \frac{ \partial }{ \partial \theta_0   } \E \left( f ( x_\theta  (T) , \hat{U}_{ \theta_0 } (T)  ) \right) \right\vert_{ \theta_0 = \theta },
\end{align}
where process $x_\theta$ is as in \eqref{defn_pdmp_param} but process $ \hat{U}_{ \theta_0 } $ is given by 
\begin{align*}
\hat{U}_{\theta_0}(t) & = U_0 + \sum_{k \in \mathcal{R}_d} Y_k \left( \int_{0}^t \hat{\lambda}_k( x_\theta(s), \hat{U}_{\theta_0}(s),\theta_0 ) ds  \right) \zeta^{(d)}_k,
\end{align*}
with $Y_k$-s being the \emph{same} unit rate Poisson processes as those in \eqref{defn_pdmp_param} and the propensity function $\hat{ \lambda}_k$ for each discrete reaction $k \in \mathcal{R}_d$ transformed according to
$$ \hat{\lambda}_k( x_\theta(s),u,\theta_0 ) :=  \lambda_k( x_\theta(s), u,\theta_0 ) + (\theta_0 - \theta) \left \langle \nabla \lambda_k( x_\theta(s), u ,\theta) , y_\theta(s) \right\rangle.$$
Such a transformation ensures that $ \hat{\lambda}_k$ matches the original propensity function $\lambda_k$ at $\theta_0 = \theta$ and
\begin{align*}
\left. \partial_{\theta_0} \hat{\lambda}_k( x_\theta(s),u,\theta_0 ) \right\vert_{ \theta_0 = \theta } = D_\theta \lambda_k(x_\theta(s) , u, \theta ).
\end{align*}• 
holds, which is required for the equivalence \eqref{result_equiv}.

As mentioned in Section \ref{sec:intro}, the main difficulty in proving Theorems \ref{thm:convergence} and \ref{thm:representation} arises due to interactions between discrete and continuous parts of the dynamics, through the propensity functions. These interactions create intricate $\theta$-dependencies which need to be carefully disentangled to show the convergence \eqref{sens_conv_result} and that PDMP sensitivity has the simple form \eqref{sens_pdmprep1}. For this we construct a coupling between the PDMP $Z_\theta = (x_\theta,U_\theta)$ and its perturbed version $Z_{\theta+h} = (x_{\theta +h},U_{\theta+h})$ and then carefully study how their difference evolves with time. See the Appendix for more details.

We remark here that computational savings that result from working with the PDMP model, instead of the original multiscale model, arise from treating as many reactions as possible \emph{continuously}. The separable form \eqref{sens_pdmprep1} of PDMP sensitivity in Theorem \ref{thm:representation} clearly demonstrates that sensitivity estimation also benefits from treating more reactions continuously, as computing the continuous sensitivity $ \hat{S}^{(c)}_\theta(f,T)$ is far easier than estimating the discrete sensitivity $\hat{S}^{(d)}_\theta(f,T)$.

We end this section with a simple algorithm that summarizes the steps needed to estimate parameter sensitivity for multiscale stochastic networks using PDMP approximations.
\begin{algorithm}[H]  
\caption{Work-flow for estimating parameter sensitivity for multiscale stochastic networks with PDMP model reductions.}      
 \label{algo_pdmp_sim_sens}
 \begin{algorithmic}[1]
\State Describe the multiscale model in terms of scaling constants $\alpha_i$-s and $\beta_k$-s (see Section \ref{sec:multiscalemodels}). Also pick an observation timescale $\gamma$.
\State If necessary, eliminate the fast reactions among discrete species by applying the quasi-stationary assumption (QSA) (see Remark \ref{rem:qssa}).
\State Derive the PDMP model reduction (see Section \ref{sec:pdmpconv}) and identify the limiting PDMP $Z_\theta = (x_\theta,U_\theta )$ \eqref{defn_pdmp_param}.
\State Simulate this PDMP (using Algorithm \ref{algo_pdmp_sim} for example) along with the process $y_\theta$ defined by IVP \eqref{defn_y_thetat}.
\State Estimate the PDMP sensitivity by separately estimating the contribution of the continuous part $ \hat{S}^{(c)}_\theta(f,T) $ and the contribution of the discrete part $ \hat{S}^{(d)}_\theta(f,T) $ (see Theorem \ref{thm:representation}). 
\end{algorithmic}
\end{algorithm}

\section{Computational examples} \label{sec:example}

In this section we illustrate our results with a couple of computational examples. In each of these examples we consider a multiscale network with CTMC dynamics, that permits a PDMP model reduction. We demonstrate that parameter sensitivities for the original multiscale CTMC dynamics can be reliably estimated from the reduced PDMP model and a small fraction of the computational costs. These costs are measured in terms of the \emph{central processing unit} (CPU) time, required for estimating the quantity of interest using an estimator implemented in C++ and executed on an Apple machine with 2.9 GHz Intel Core i5 processor.

\subsection{A multiscale gene-expression network} \label{sec:example1}

As our first example we consider the multiscale gene-expression network \cite{MO}. Here a gene ($G$) transcribes mRNA ($M$) molecules that then translate into the corresponding protein ($P$) molecules (see Figure \ref{figure:gefig}{\bf A}). Both $M$ and $P$ molecules undergo first-order degradation at certain rates, and we assume that there is a negative transcriptional feedback from the protein molecules which is modeled as a Hill-type propensity function. The gene $G$ can switch between active ($G_{ \textnormal{on} }$) and inactive ($G_{ \textnormal{off} }$) states and transcription is only possible in the active state $G_{ \textnormal{on} }$. The reactions for this network are given in Table \ref{rxnchart:geneexnetwork}.

\begin{table}[h!]
\begin{align*}
\begin{array}{|c|c|c|} \hline
\textnormal{Reaction No.} &  \textnormal{Reaction} & \textnormal{Propensity  } \lambda'_k(x_1,x_2,x_3) \\ \hline
{\bf 1}& G_{ \textnormal{off} }  \longrightarrow G_{ \textnormal{on} }& N^{2/3}_0(1 - x_3)\\
{\bf 2}& G_{ \textnormal{on}}  \longrightarrow G_{ \textnormal{off}}& N^{2/3}_0 x_3\\
3& G_{ \textnormal{on}}  \longrightarrow G_{ \textnormal{on}} + M& \frac{20 N_0 x_3}{N_0 +  \theta_1 x_2}\\
4& M \longrightarrow M + P&  N_0 \theta_2 x_1\\
5& M  \longrightarrow \emptyset & \theta_3 x_1\\
6& P \longrightarrow \emptyset&  \theta_4 x_2\\ \hline
\end{array}
\end{align*}
\caption{Reactions for the multiscale gene-expression network. The associated propensity functions $\lambda'_k$-s are also provided. Here $x_1$, $x_2$ and $x_3$ denote the number of mRNA counts, protein counts and ON gene-copies respectively.  Note that $x_3 \in \{0,1\}$. We set $N_0 =1000$ and select parameters as $\theta_1 =1 ,\theta_2 = 0.01, \theta_3 = 1$ and $\theta_4 = 0.1$. Reactions $1$ and $2$ are much \emph{faster} in comparison to the other reactions.
}
\label{rxnchart:geneexnetwork}
\end{table}

Henceforth the species are ordered as ${\bf S}_1 = M$, ${\bf S}_2 = P$ and ${\bf S}_3 = G_{ \textnormal{on} }$. Note that the copy-number of $G_{ \textnormal{on} }$ is either $0$ or $1$. We suppose that proteins are abundant but mRNAs are present in small numbers. Therefore we choose the species abundance factors as $\alpha_1 = 0$, $\alpha_2 = 1$ and $\alpha_3 = 0$ (see Section \ref{sec:multiscalemodels}). We set $\beta_k = 0$ for reactions $k=3,4,5,6$ and $\beta_k = 2/3$ for reactions $k =1,2$. The observation timescale is selected as $\gamma = 0$ and the initial state is $(0,0,1)$.

It is not possible to apply the PDMP model reduction, as described in Section \ref{sec:pdmpconv}, directly on this multiscale network due to the presence of fast reactions (reactions 1 and 2) among discrete species. However we can first apply the quasi-stationary assumption (QSA) to eliminate these reactions (see Remark \ref{rem:qssa}) and then apply the PDMP model reduction. In this example, application of QSA is simple as the stationary distribution for the copy-numbers of species $ G_{ \textnormal{on} }$ is just Bernoulli with probability $p_{ \textnormal{eq} } = 1/2$. After applying QSA and the PDMP model reduction we arrive at a PDMP which approximately captures the original multiscale stochastic dynamics under species scalings specified by $\alpha_i$-s. This PDMP is described in Table \ref{rxnchart:pdmp_geneexnetwork}.
\begin{table}[h!]
\begin{align*}
\begin{array}{|c|c|c|c|} \hline
\textnormal{Reaction No.} & \textnormal{Reaction type} &  \textnormal{Reaction} & \textnormal{Propensity  } \lambda_k(z_1,z_2) \\ \hline
3&  \textnormal{discrete} & \emptyset  \longrightarrow  M& \frac{20 p_{ \textnormal{eq} } }{1 + \theta_1 z_2}\\
4& \textnormal{continuous}&M \longrightarrow M + P&  \theta_2 z_1\\
5& \textnormal{discrete} &M  \longrightarrow \emptyset & \theta_3  z_1\\
6&\textnormal{continuous} &P \longrightarrow \emptyset&  \theta_4 z_2 \\ \hline
\end{array}
\end{align*}
\caption{PDMP model reduction for the multiscale gene-expression network. This PDMP has one discrete species $\mathcal{S}_d = \{M\}$, one continuous species $\mathcal{S}_c = \{P\}$, two discrete reactions $\mathcal{R}_d =\{3,5\}$ and two continuous reactions $\mathcal{R}_c = \{4,6\}$. The associated propensity functions $\lambda_k$-s are provided. The parameters ($\theta_i$-s) are the same as in Table \ref{rxnchart:geneexnetwork}. Here $z_1$ is the copy-number of mRNA ($M$) and $z_2$ is the concentration of protein ($P$). 
}
\label{rxnchart:pdmp_geneexnetwork}
\end{table}

We simulate $10^5$ trajectories of this PDMP with Algorithm \ref{algo_pdmp_sim} in the time-period $[0,T]$ with $T = 50$, and for comparison we also simulate $10^5$ trajectories of the CTMC describing the original multiscale reaction network with Gillespie's SSA \cite{GP}. As we observe in Figures \ref{figure:gefig}{\bf C} and \ref{figure:gefig}{\bf D}, the match between the two approaches is quite good, which demonstrates that the limiting PDMP captures the original dynamics very accurately. However, as expected, the computational costs for PDMP simulation is less than $20\%$ of the cost of SSA simulation (see Figure \ref{figure:gefig}{\bf B}).

\begin{figure}[ht!]
\centering
\frame{\includegraphics[width=0.98\textwidth]{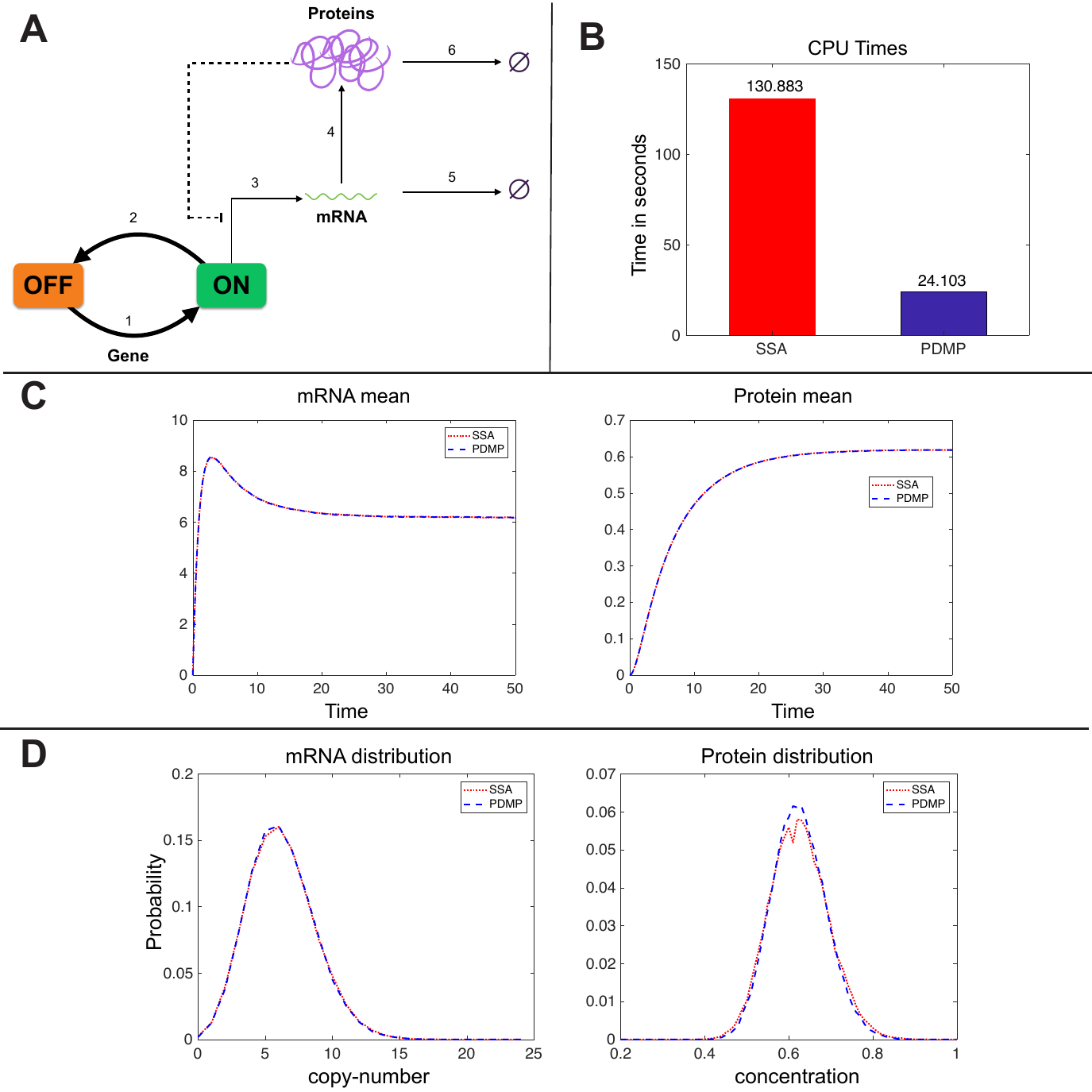}}
\caption{Panel {\bf A} is a cartoon of the multiscale gene-expression network with transcriptional feedback from the protein molecules. The thick arrows indicate the fast reactions among discrete species that can be eliminated with QSA application. Panel {\bf B} compares the computational costs for simulating $10^5$ trajectories for the original multiscale CTMC dynamics with Gillespie's SSA and for the reduced PDMP dynamics with Algorithm \ref{algo_pdmp_sim}. Panel {\bf C} provides a comparison between SSA and PDMP, for the mean dynamics of mRNAs and proteins in the time-period $[0,50]$ and panel {\bf D} provides the same comparison for the two marginal distributions at time $T = 50$. }
\label{figure:gefig} 
\end{figure}

Next we compute sensitivities for the expected protein concentration at time $T = 50$, w.r.t. model parameters $\theta_1,\dots,\theta_4$ whose values are as in Table \ref{rxnchart:geneexnetwork}. For this purpose we use the \emph{Integral Path Algorithm} (IPA) \cite{gupta2017estimation} with $10^4$ samples, obtained by simulation of the original multiscale dynamics with Gillespie's SSA and the PDMP dynamics with Algorithm \ref{algo_pdmp_sim}. Note that for PDMP dynamics, IPA is only used to estimate the sensitivity  contribution of the discrete part $\hat{S}^{(d)}_\theta(f,T) $ (see Theorem \ref{thm:representation}). The results of this sensitivity analysis are shown in Figure \ref{figure:ge_sensitivity_values}. While both the estimation approaches SSA-IPA (i.e.\ IPA with SSA simulations) and PDMP-IPA (i.e.\ IPA with PDMP simulations) provided very similar sensitivity values, the PDMP-based approach is around 10 times faster.

\begin{figure}[ht!]
\centering
\includegraphics[width=0.98\textwidth]{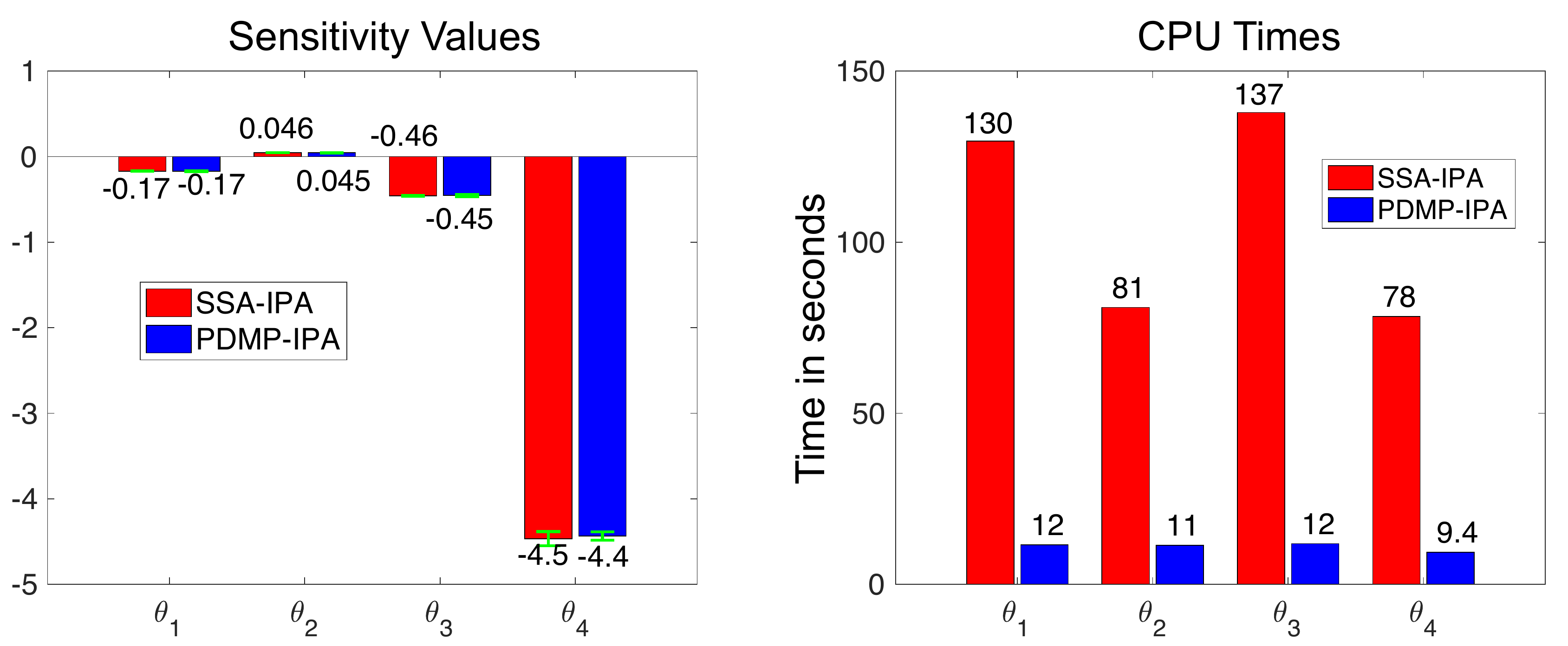}
\caption{Sensitivity analysis for the multiscale gene-expression network. The left bar-chart depicts the estimated sensitivities for the expected protein concentration at time $T = 50$, w.r.t. model parameters $\theta_1,\dots,\theta_4$. These values are estimated with SSA-IPA and PDMP-IPA using $10^4$ samples. The \emph{error-bars} signify the estimator standard deviation. The right bar-chart provides a comparison of the computational costs associated with sensitivity estimation w.r.t.\ all the parameters, with both the approaches SSA-IPA and PDMP-IPA.   
}
\label{figure:ge_sensitivity_values} 
\end{figure}

\subsection{Michaelis-Menten Enzyme Kinetics} \label{secex:mm}

We now consider the standard Michaelis-Menten enzyme kinetics model \cite{michaelis2007kinetik}, where an \emph{enzyme} ($E$) catalyzes the conversion of \emph{substrate} ($S$) into \emph{product} ($P$), by forming an intermediate \emph{complex} ($ES$) via reversible binding between the enzyme and the substrate. This model can be schematically represented as  
\begin{align}
\label{scheme_mm}
E +S \xrightleftharpoons[ \kappa_2]{ \kappa_1} ES \stackrel{\kappa_3}{\longrightarrow} E + P.
\end{align}
Each of the three reactions follow mass-action kinetics \eqref{defn:massactionkinetics} and $\kappa_i$ is the rate constant for reaction $i$. We order the species as ${\bf S}_1 = S$, ${\bf S}_2 = P$, ${\bf S}_3 = E$ and ${\bf S}_4 = ES$. We suppose that substrate/product molecules are in high abundance while enzyme is only present in small copy-numbers bounded by a positive integer $M$. This ensures that copy-numbers of complex $ES$ are also bounded by $M$ due to the conservation relation
\begin{align*}
x_3 + x_4  = M
\end{align*}
where $x_3$ and $x_4$ denote the copy-numbers of $E$ and $ES$ respectively. We choose the species abundance factors as $\alpha_1 =\alpha_2 =1$ and $\alpha_3 = \alpha_4 = 0$ (see Section \ref{sec:multiscalemodels}). Moreover we set $\beta_1 = 0$ for reaction $1$ and $\beta_i=1$ for reactions $i=2,3$ allowing us to express the rate constants $\kappa_i$-s in terms of normalized parameters $\theta_i$-s and the volume scaling parameter $N_0$ as
\begin{align*}
\kappa_1 = \theta_1, \quad \kappa_2 = N_0 \theta_2 \quad \textnormal{and} \quad \kappa_3 = N_0 \theta_3. 
\end{align*}
The observation timescale is selected as $\gamma = 0$ and the initial state is chosen to be $(0.5 N_0,0,M,0)$.

Under the scaling we just specified, one can apply the quasi-stationary assumption (QSA) (see Remark \ref{rem:qssa}) to eliminate the \emph{fast} reversible binding-unbinding reactions
\begin{align}
\label{MMQSANetwork}
E +S \xrightleftharpoons[ N_0 \theta_2]{ N_0 \theta_1} ES 
\end{align}
and obtain the reduced network which simply consists of a single conversion reaction
\begin{align}
\label{MMsimplnetwork}
S \stackrel{ \lambda_\theta(x_1)}{\longrightarrow}  P
\end{align}
where $\theta = (\theta_1,\theta_2,\theta_3)$ is the parameter vector, $x_1$ is the substrate \emph{concentration} and $\lambda_\theta(x_1)$ is the Michaelis-Menten rate function given by
\begin{align*}
\lambda_\theta(x_1) = \frac{M \theta_1 \theta_3 x_1 }{ \theta_2 + \theta_3 + \theta_1 x_1 }.
\end{align*}
To apply QSA on subnetwork \eqref{MMQSANetwork} one notes that for a fixed substrate concentration $x_1$, the stationary distribution for enzyme copy-numbers is binomial with parameters $M$ and
\begin{align*}
p = \frac{\theta_2 + \theta_3}{\theta_2 + \theta_3 + \theta_1 x_1}.
\end{align*}
The ODE dynamics $(x_1(t) , x_2(t) )_{t \geq 0}$ for the reduced network \eqref{MMsimplnetwork} is given by
\begin{align}
\label{mmreducedodedynamics}
\frac{dx_1}{dt} = -  \lambda_\theta(x_1) \quad \textnormal{and} \quad \frac{dx_2}{dt} =   \lambda_\theta(x_1)
\end{align}
where $x_1(t)$ and $x_2(t)$ denote the concentrations at time $t$ of substrate and protein respectively. For more details on this Michaelis-Menten model reduction we refer the readers to \cite{darden1979pseudo} or Section 6.4 in \cite{HWKang}.

\begin{figure}[ht!]
\centering
\includegraphics[width=0.98\textwidth]{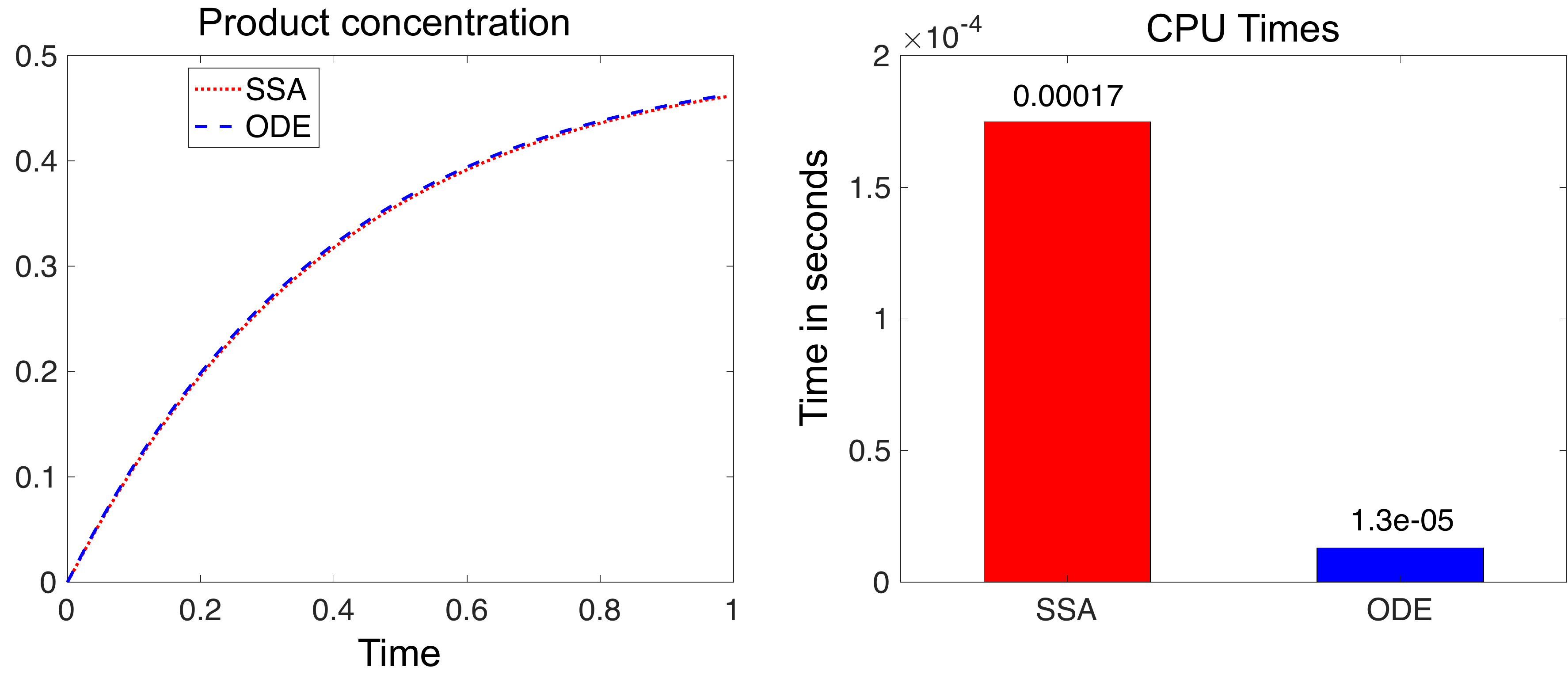}
\caption{ODE approximation of the multiscale Michaelis-Menten network kinetics. The figure on the left provides a comparison between SSA and ODE, for the mean dynamics of Product concentration in the time-period $[0,1]$. The mean dynamics is computed with $10^4$ SSA trajectories for the original model and with a single trajectory for the reduced ODE model. The figure on the right compares the computational costs for simulating a single trajectory for the original dynamics with SSA and for the reduced ODE dynamics with Algorithm \ref{algo_pdmp_sim}.}
\label{figure:MM1} 
\end{figure}

From now on we set the model parameters as $N_0 =1000$, $M =20$, $\theta_1 = 0.3$, $\theta_2 = 1$ and $\theta_3 = 0.8$. In Figure \ref{figure:MM1} we demonstrate the accuracy of the reduced ODE model in capturing the mean dynamics of Product concentration for the multiscale Michaelis-Menten network \eqref{scheme_mm}. The ODE dynamics \eqref{mmreducedodedynamics} is simulated with Algorithm \ref{algo_pdmp_sim}, and since there are no discrete reactions, this algorithm simply becomes the standard ODE-solver based on an explicit Euler scheme. It can be seen from Figure \ref{figure:MM1} that the computational costs for ODE simulation are only about $7.65\%$ of the computational costs for SSA simulations of the original multiscale stochastic model.

We now compute the sensitivities for the expected product concentration at time $T=1$ w.r.t. parameters $\theta_1,\theta_2$ and $\theta_3$. For the original multiscale CTMC dynamics we compute sensitivities using IPA with $10^4$ trajectories obtained with Gillespie's SSA. For the limiting ODE dynamics \eqref{mmreducedodedynamics}, the sensitivities can simply be computed by solving the ODE \eqref{defn_y_thetat} and evaluating the sensitivity contribution $ \hat{S}^{(c)}_\theta(f,T) $ of the continuous dynamics (see Theorem \ref{thm:representation}). The results of this sensitivity analysis are shown in Figure \ref{figure:MM2}. While both the estimation approaches provided very similar sensitivity values, the ODE-based approach is around 12 times faster.

\begin{figure}[ht!]
\centering
\includegraphics[width=0.98\textwidth]{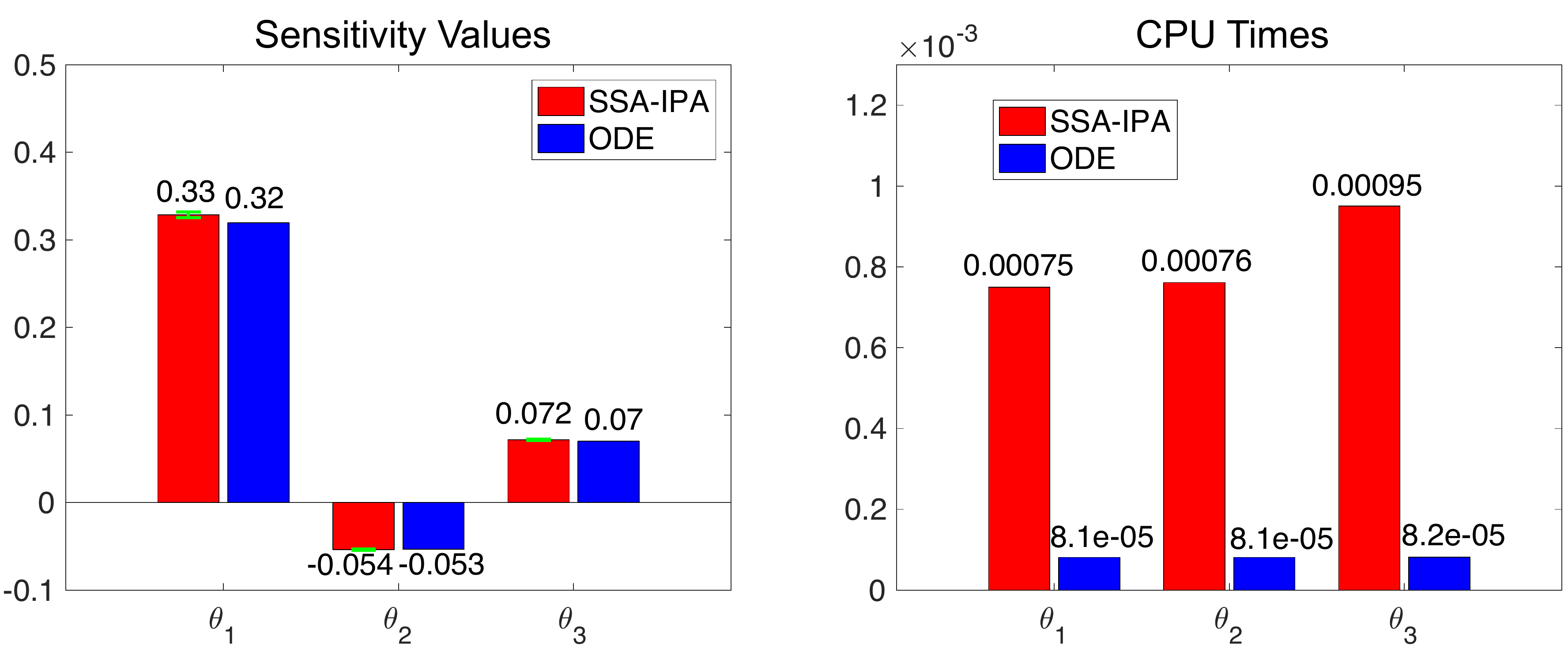}
\caption{Sensitivity analysis for the Michaelis-Menten network kinetics. The left bar-chart depicts the estimated sensitivities for the mean product concentration at time $T = 1$, w.r.t. model parameters $\theta_1,\theta_2$ and $\theta_3$. These values are estimated for the original CTMC dynamics with SSA-IPA using $10^4$ samples and the error bar signifying the estimator standard deviation is also shown. For the limiting ODE model, the sensitvities are simply computed by solving the ODE \eqref{defn_y_thetat}. The right bar-chart provides a comparison of the computational costs associated with generating each sensitivity sample, with both the approaches.
}
\label{figure:MM2} 
\end{figure}

\section{Conclusion} \label{sec:concl}

Sensitivity analysis of stochastic models of multiscale reaction networks is extremely difficult because existing sensitivity estimation methods generally require simulations of the stochastic dynamics which is prohibitively expensive for multiscale networks. Recently it has been shown that often for multiscale networks, the complexity can be reduced by approximating the dynamics as a \emph{Piecewise Deterministic Markov process} (PDMP), which is a hybrid process consisting of both discrete and continuous components \cite{crudu2009hybrid,HWKang}. Our aim in this paper is to demonstrate that such PDMP approximations can provide accurate estimates of the parameter sensitivity for the original multiscale model at a fraction of the computational costs. To this end we prove two main results. The first result (Theorem \ref{thm:convergence}) proves that in the limit where the PDMP approximation becomes exact, the original sensitivity converges to the corresponding PDMP sensitivity. The second result (Theorem \ref{thm:representation}) provides a representation of the PDMP parameter sensitivity that separates the contributions of discrete and continuous components in the dynamics. We discuss how each of these contributions can be efficiently estimated.

Observe that PDMP convergence of the dynamics of the stochastic model of a multiscale reaction network, is essentially a generalization of the classical thermodynamic limit result of Kurtz \cite{kurtz1978strong}, which stipulates convergence of the stochastic model to the ODE-based deterministic model of the reaction network, under a suitable scaling of reaction-rate constants, when all the species are uniformly abundant, i.e. all species are of order $N$ or $O(N)$. In such a setting, our results in this paper show that along with the dynamics, the sensitivities w.r.t.\ various model parameters will also converge to the corresponding sensitivities for the limiting deterministic model (see the example in Section \ref{secex:mm}). Even though this paper is motivated by applications in Systems Biology, estimation of PDMP parameter sensitivities via Theorem \ref{thm:representation} could be useful in other scientific areas such as Reliability Theory \cite{Eymard2011}.

The approach in this paper would benefit significantly if the convergence analysis of Theorem \ref{thm:convergence} can be extended to obtain concrete estimates of the approximation error using ideas similar to those in \cite{ganguly2015jump}. Another promising future research direction would be to integrate the results in this paper with approaches that automate PDMP model reductions for multiscale reaction networks \cite{hepp2015adaptive}. It would also be beneficial to extend the results in this paper to allow for diffusion in the continuous component of the dynamics i.e.\ between the jumps of the discrete component, the continuous component evolves like a Stochastic Differential Equation (SDE) rather than an ODE. It has been shown that such jump-diffusion models are often more accurate than PDMP models in capturing the stochastic dynamics of multiscale models \cite{crudu2009hybrid}.

\section*{Appendix} 

In this appendix we prove the main results of the paper which are Theorems \ref{thm:convergence} and \ref{thm:representation}. Throughout this section we use the same notation as in Section \ref{sec:mainresults}. In particular, $\Psi_{t}(x,U,\theta)$ is defined by \eqref{defn_psi} and under Assumption \ref{assmp:key}, this function is differentiable w.r.t.\ $x$. We start this section with a simple proposition.
\begin{proposition}
\label{prop:convergence}
Let the multiscale process $( Z^N_\theta(t) )_{t  \geq 0}$ and the PDMP $( Z_\theta(t) )_{t  \geq 0}$ be as in Section \ref{sec:mainresults}. Suppose that Assumption \ref{assmp:key} holds and $Z^N_\theta \Rightarrow Z_\theta$ as $N  \to \infty$. Then
\begin{align*}
\lim_{N  \to \infty}  \frac{ \partial }{ \partial \theta} \E( f( Z^{N}_\theta(  T ) ) )  = \bar{S}_\theta(f,T)
\end{align*}
where
\begin{align}
\label{defn_barstheta}
\bar{S}_\theta(f,T) &=  \sum_{k \in  \mathcal{R}_c}   \E\left[ \int_{0}^T   \partial_\theta \lambda_k   (x_\theta(t) ,U_\theta(t) ,\theta ) \left\langle  \nabla \Psi_{T-t}(x_\theta(t),U_\theta(t),T-t) , \zeta^{(c)}_k  \right\rangle dt \right] \\
& +  \sum_{k \in \mathcal{R}_d} \E  \left[ \int_{0}^T   \partial_\theta \lambda_k   (x_\theta(t) ,U_\theta(t) ,\theta ) \Delta_{k} \Psi_{T-t}(x_\theta(t),U_\theta(t),T-t)  dt \right]. \notag
\end{align}
\end{proposition}
\begin{proof}
Let 
\begin{align*}
S^N_\theta(f,T) =   \frac{ \partial }{ \partial \theta} \E( f( Z^{N}_\theta(  T ) ) ) 
\end{align*}
and analogous to $\Psi_t$ \eqref{defn_psi} define the map $\Psi^N_t$ as
\begin{align*}
\Psi^N_t (x,U,\theta) = \E( f( Z^N_\theta (t) ) ), \quad \textnormal{for any} \quad t \geq 0,
\end{align*}
where $Z^N$ is the scaled process describing the multiscale reaction dynamics with $(x,U)$ as its initial state. Due to Theorem 3.1 in \cite{gupta2017estimation} we obtain
\begin{align*}
& S^N_\theta(f,T)  \\
&= \sum_{ k =1}^K \E \left( N^{\rho_k + r } \int_0^T \partial_\theta  \lambda^N_k( Z^N_\theta(t) ,\theta ) (  \Psi^N_{T-t}( Z^N_\theta(t)+ \zeta^N_k,\theta )  - \Psi^N_{T-t}( Z^N_\theta(t) , \theta  ) ) dt   \right),
\end{align*}
where $\zeta^N_k:= \Lambda_N \zeta_k$, $\rho_k = \beta_k + \langle \nu_k,  \alpha \rangle$ and $r$ is the timescale of observation \eqref{pdmp_conv_timescale}. We can write $Z^N_\theta(t) = (x^N_\theta(t), U^N_\theta(t) )$ where $x^N_\theta(t) \in \R^{S_c}$ denotes the states of species in $\mathcal{S}_c$ and  $U^N_\theta(t) \in \N_0^{S_d}$ denotes the states of species in $\mathcal{S}_d$. Exploiting the analysis in Section \ref{sec:pdmpconv} we can express $S^N_\theta(f,T) $ as
\begin{align*}
S^N_\theta(f,T) = S^{N,c}_\theta(f,T)  + S^{N,d}_\theta(f,T) +o(1)
\end{align*}
where the $o(1)$ term converges to $0$ as $N  \to \infty$,
\begin{align*}
S^{N,c}_\theta(f,T) & = \sum_{  k \in \mathcal{R}_c} \E \left(  \int_0^T \partial_\theta  \lambda_k( x^N_\theta(t), U^N_\theta(t) ,\theta ) \right. \\ &   \left. N^{\rho_k + r } (  \Psi^N_{T-t}( x^N_\theta(t)+N^{-(\rho_k + r) } \zeta^{(c)}_k , U^N_\theta(t) , \theta )  - \Psi^N_{T-t}( x^N_\theta(t), U^N_\theta(t) ,\theta  ) ) dt   \right) 
\end{align*}
and
\begin{align*}
S^{N,d}_\theta(f,T) &=  \sum_{  k \in \mathcal{R}_d}\E \left(  \int_0^T \partial_\theta  \lambda_k( x^N_\theta(t), U^N_\theta(t)  ,\theta )  \right. \\ & \qquad \left.  (  \Psi^N_{T-t}( x^N_\theta(t), U^N_\theta(t)+ \zeta^{(d)}_k,\theta )  - \Psi^N_{T-t}( x^N_\theta(t), U^N_\theta(t) , \theta  ) ) dt   \right).
\end{align*}

We know that as $N  \to \infty$, process $(x^N_\theta, U^N_\theta)$ converges in distribution to process $(x_\theta, U_\theta)$ in the Skorohod topology on $\R^{S_c} \times \N^{S_d}_0$. This ensures that for any $(x,U)$ and $t \geq 0$, $\Psi^N_t(x,U,\theta) \to \Psi_t(x,U,\theta)$ as $N  \to \infty$, and this convergence holds uniformly over compact sets, i.e.
\begin{align}
\label{compactconv1}
\lim_{N  \to \infty} \sup_{ (x,U) \in C , t \in [0,T]  } \left| \Psi^N_t(x,U,\theta) - \Psi_t(x,U,\theta)   \right| = 0
\end{align}
for any $T >0$ and any compact set $C \subset \R^{S_c} \times \N^{S_d}_0$. In fact, under Assumptions \ref{assmp:key} we also have
\begin{align}
\label{compactconv2}
\lim_{N  \to \infty} \sup_{ (x,U) \in C , t \in [0,T]  } \left\|  \nabla \Psi^N_t(x,U,\theta) -  \nabla \Psi_t(x,U,\theta)   \right\| = 0.
\end{align}

As $(x^N_\theta, U^N_\theta)  \Rightarrow (x_\theta, U_\theta)$, using \eqref{compactconv1} it is straightforward to conclude that
\begin{align}
\label{propstart_firstlim}
\lim_{N \to \infty} S^{N,d}_\theta(f,T)    = \sum_{k \in \mathcal{R}_d} \E  \left[ \int_{0}^T   \partial_\theta \lambda_k   (x_\theta(t) ,U_\theta(t) ,\theta ) \Delta_{k} \Psi_{T-t}(x_\theta(t),U_\theta(t),T-t)  dt \right].
\end{align}
Noting that 
\begin{align*}
& N^{\rho_k + r } \left[  \Psi^N_{T-t}\left( x^N_\theta(t)+ \frac{1}{ N^{\rho_k + r } }\zeta^{(c)}_k , U^N_\theta(t) , \theta \right)  - \Psi^N_{T-t}\left( x^N_\theta(t), U^N_\theta(t) ,\theta  \right) \right] \\
& =  \left\langle \nabla \Psi^N_{T-t}\left( x^N_\theta(t), U^N_\theta(t) ,\theta  \right), \zeta^{(c)}_k \right\rangle +o(1),
\end{align*}
\eqref{compactconv2} allows us to obtain
\begin{align*}
& \lim_{N \to \infty} S^{N,c}_\theta(f,T) \\
&=\sum_{k \in  \mathcal{R}_c}   \E\left[ \int_{0}^T   \partial_\theta \lambda_k   (x_\theta(t) ,U_\theta(t) ,\theta ) \left\langle  \nabla \Psi_{T-t}(x_\theta(t),U_\theta(t),T-t) , \zeta^{(c)}_k  \right\rangle dt \right]. 
\end{align*}
This relation along with \eqref{propstart_firstlim} proves the proposition.
\end{proof}

In light of Proposition \ref{prop:convergence}, to prove Theorem \ref{thm:convergence} it suffices to show that $\bar{S}_\theta(f,T)  = \hat{S}_\theta(f,T) $, where $\hat{S}_\theta(f,T) $ is the sensitivity for limiting PDMP $(Z_\theta(t))_{t \geq 0 }$ defined by
\begin{align}
\label{defn_pdmp_sensitivity}
\hat{S}_\theta(f,T) &= \lim_{ h \to \infty} \frac{ \E( f(Z_{\theta+h} (T) ) )  - \E( f(Z_{\theta} (T) ) ) }{h} \\
&= \lim_{ h \to \infty} \frac{ \E( f(x_{\theta+h} (T),  U_{\theta+h} (T) ) )  - \E( f(x_{\theta} (T),  U_{\theta} (T) )  ) }{h}. \notag
\end{align}
The next proposition derives a formula for $\hat{S}_\theta(f,T) $ by coupling processes $Z_\theta = (x_\theta, U_\theta)$ and $Z_{\theta +h} = (x_{\theta +h}, U_{\theta +h} )$. This formula will be useful later in proving both Theorems \ref{thm:convergence} and \ref{thm:representation}.

\begin{proposition}
\label{prop:mainprop}
Let $y_\theta(t)$ be the solution of IVP \eqref{defn_y_thetat} and let $D_\theta \lambda_k(  x_\theta(t), U_\theta(t) ,\theta)$ be given by \eqref{defn_part_tod_der}. Then the PDMP sensitivity $\hat{S}_\theta(f,T) $ defined by \eqref{defn_pdmp_sensitivity} can be expressed as 
\label{prop:sensrep1}
\begin{align*}
\hat{S}_\theta(f,T)  &= \sum_{k \in \mathcal{R}_c} \E\left[  \int_{0}^{T} \partial_{\theta} \lambda_k (x_\theta(t) ,U_\theta(t) ,\theta ) \left\langle  \nabla  f(x_\theta(t),U_\theta(t) ) , \zeta^{(c)}_k  \right\rangle dt   \right] \\
&+ \sum_{k \in \mathcal{R}_c} \E\left[  \int_{0}^{T}  \left \langle  \nabla \left[   \lambda_k (x_\theta(t) ,U_\theta(t) ,\theta ) \left\langle f(x_\theta(t),U_\theta(t) ) , \zeta^{(c)}_k  \right\rangle \right] , y_\theta(t)   \right\rangle dt   \right]  \\
&+ \sum_{k \in \mathcal{R}_d} \E\left[  \int_{0}^{T}  \lambda_k (x_\theta(t) ,U_\theta(t) ,\theta )  \left \langle  \nabla  \left( \Delta_k f(x_\theta(t),U_\theta(t) )   \right)  , y_\theta(t)   \right\rangle dt   \right] \\
&+   \sum_{k \in \mathcal{R}_d } \E \left[ \int_0^T   \partial_{\theta}   \lambda_k (x_\theta(t) , U_\theta(t) ,\theta )    \Delta_k \Psi_{T- t}( x_\theta(t), U_\theta(t), \theta ) dt  \right]\\
&+   \sum_{k \in \mathcal{R}_d } \E \left[ \int_0^T   \left\langle \nabla \lambda_k (x_\theta(t) , U_\theta(t) ,\theta ), y_\theta(t)   \right\rangle \Delta_k \Psi_{T- t}( x_\theta(t), U_\theta(t), \theta ) dt  \right].
\end{align*}
\end{proposition}

\begin{proof}
Analogous to the ``split-coupling" introduced in \cite{DA} we couple the PDMPs $Z_\theta = (x_\theta, U_\theta)$ and $Z_{\theta+h} =  (x_{\theta+h} , U_{\theta+h} )$ as below
\begin{align*}
 &x_\theta(t)= x_0 + \sum_{k \in \mathcal{R}_c} \left( \int_{0}^t \lambda_k( x_\theta(s), U_\theta(s), \theta ) ds \right) \zeta^{(c)}_k \\
&x_{\theta+h}(t) = x_0 + \sum_{k \in \mathcal{R}_c} \left( \int_{0}^t \lambda_k( x_{\theta+h}(s), U_{\theta+h}(s), \theta +h ) ds \right) \zeta^{(c)}_k \\
&U_\theta(t)  = U_0 \\
&+ \sum_{k \in \mathcal{R}_d} Y_k \left( \int_{0}^t \lambda_k( x_\theta(s), U_\theta(s) ,\theta ) \wedge \lambda_k( x_{\theta+h}(s), U_{\theta+h}(s), \theta +h )   ds  \right) \zeta^{(d)}_k \\
& +  \sum_{k \in \mathcal{R}_d} Y^{(1)}_k \left( \int_{0}^t \lambda^{(1)}_k( x_\theta(s), U_\theta(s) ,\theta , x_{\theta+h}(s), U_{\theta+h}(s), \theta +h )   ds  \right) \zeta^{(d)}_k \\
&U_{\theta+h}(t)  = U_0\\
& + \sum_{k \in \mathcal{R}_d} Y_k \left( \int_{0}^t \lambda_k( x_\theta(s), U_\theta(s) ,\theta ) \wedge \lambda_k( x_{\theta+h}(s), U_{\theta+h}(s), \theta +h )   ds  \right) \zeta^{(d)}_k \\
& +  \sum_{k \in \mathcal{R}_d} Y^{(2)}_k \left( \int_{0}^t \lambda^{(2)}_k( x_\theta(s), U_\theta(s) ,\theta , x_{\theta+h}(s), U_{\theta+h}(s), \theta +h ) ds  \right) \zeta^{(d)}_k, 
\end{align*}
where $a \wedge b$ denotes the minimum of $a$ and $b$, $\{ Y_k , Y^{ (1) }_k , Y^{ (2)}_{k}\}$ is a collection of independent unit rate Poisson processes, and
\begin{align*}
&\lambda^{(1)}_k( x_\theta(s), U_\theta(s) ,\theta , x_{\theta+h}(s), U_{\theta+h}(s), \theta +h ) = \lambda_k( x_\theta(s), U_\theta(s) ,\theta )   \\
&-  \lambda_k( x_\theta(s), U_\theta(s) ,\theta ) \wedge \lambda_k( x_{\theta+h}(s), U_{\theta+h}(s), \theta +h ) \quad \textnormal{and} \\
 &  \lambda^{(2)}_k( x_\theta(s), U_\theta(s) ,\theta , x_{\theta+h}(s), U_{\theta+h}(s), \theta +h ) = \lambda_k( x_{\theta+h}(s), U_{\theta+h}(s) ,\theta +h ) \\
& -  \lambda_k( x_\theta(s), U_\theta(s) ,\theta ) \wedge \lambda_k( x_{\theta+h}(s), U_{\theta+h}(s), \theta +h ).
\end{align*}
Define a stopping time as the first time that processes $U_\theta$ and $U_{ \theta+h}$ separate, i.e.
\begin{align*}
\tau_h = \inf\{ t \geq 0: U_{\theta}(t) \neq U_{\theta+h}(t) \}.
\end{align*}
Observe that the generator for the PDMP $Z_\theta = (x_\theta, U_\theta)$ is
\begin{align*}
\mathbb{A}_\theta g(x,u) = \sum_{ k \in \mathcal{R}_c} \lambda_k(x,u,\theta) \left \langle \nabla g(x,u) ,  \zeta^{(c)}_k  \right \rangle  + \sum_{k \in \mathcal{R}_d} \lambda_k(x,u,\theta) \Delta_k g(x,u),
\end{align*}
where $g : \R^{S_c} \times \N^{S_d}_0$ is any function which is differentiable in the first $S_c$ coordinates. Applying Dynkin's formula we obtain
\begin{align*}
\E\left( f(x_\theta(t) ,U_\theta(t) )  \right) & = f(x_0,U_0) + \E\left( \int_{0}^t \mathbb{A}_\theta f( x_\theta(s) , U_\theta(s) ) ds \right) \\
\textnormal{and}  \quad \E\left( f(x_{\theta+h}(t) ,U_{\theta+h}(t) )  \right) & = f(x_0,U_0) + \E\left( \int_{0}^t \mathbb{A}_{\theta+h} f( x_{\theta+h}(s) , U_{\theta+h}(s) ) ds \right).
\end{align*}
The above coupling between processes $Z_\theta = (x_\theta, U_\theta)$ and $Z_{\theta+h} =  (x_{\theta+h} , U_{\theta+h} )$ ensures that for $0 \leq s \leq \tau_h$ we have $U_{\theta+h}(s) = U_{\theta}(s)$ and $x_{\theta+h}(s) = x_\theta(s) + h y_\theta(t) +o(h)$. Noting that $\tau_h \to \infty$ a.s. as $h \to 0$ we obtain
\begin{align}
\label{mainpropproof_part1}
&\lim_{h \to 0 } \frac{1}{h} \left[   \E\left( \int_{0}^{\tau_h\wedge t} \mathbb{A}_{\theta+h} f( x_{\theta+h}(s) , U_{\theta+h}(s) ) ds \right) - \E\left( \int_{0}^{\tau_h\wedge t} \mathbb{A}_{\theta} f( x_{\theta}(s) , U_{\theta}(s) ) ds \right) \right] \notag \\
& = \lim_{h \to 0 } \frac{1}{h} \left[  \E\left( \int_{0}^{\tau_h\wedge t}  \left[ \mathbb{A}_{\theta+h} f( x_{\theta+h}(s) , U_{\theta}(s) ) -   \mathbb{A}_{\theta} f( x_{\theta}(s) , U_{\theta}(s) )   \right]   ds    \right)  \right] \notag  \\
& = \sum_{k \in \mathcal{R}_c} \E\left[  \int_{0}^{ t}  \partial_\theta \lambda_k  (x_\theta(s) ,U_\theta(s) ,\theta ) \left\langle \nabla f(x_\theta(s),U_\theta(s) ) , \zeta^{(c)}_k  \right\rangle ds   \right]  \notag  \\
&+ \sum_{k \in \mathcal{R}_c} \E\left[  \int_{0}^{t}  \left \langle \nabla \left[   \lambda_k (x_\theta(s) ,U_\theta(s) ,\theta ) \left\langle \nabla f(x_\theta(s),U_\theta(s) ) , \zeta^{(c)}_k  \right\rangle \right] , y_\theta(s)   \right\rangle ds   \right]  \notag  \\
&+ \sum_{k \in \mathcal{R}_d} \E\left[  \int_{0}^{ t}  \left \langle \nabla \left[   \lambda_k (x_\theta(s) ,U_\theta(s) ,\theta ) \Delta_k f(x_\theta(s),U_\theta(s)  )  \right] , y_\theta(s)   \right\rangle ds   \right]  \notag  \\
& + \sum_{k \in \mathcal{R}_d} \E\left[  \int_{0}^{ t}  \partial_\theta \lambda_k(x_\theta(s) ,U_\theta(s) ,\theta ) \Delta_k f(x_\theta(s),U_\theta(s))    ds   \right].
\end{align}

Let $\sigma_0 = 0$ and for each $i=1,2,\dots$ let $\sigma_i$ denote the $i$-th jump time of the process
\begin{align*}
 \sum_{k \in \mathcal{R}_d} Y_k \left( \int_{0}^t \lambda_k( x_\theta(s), U_\theta(s) ,\theta ) \wedge \lambda_k( x_{\theta+h}(s), U_{\theta+h}(s), \theta +h )   ds  \right)
\end{align*}
which counts the common jump times among processes $U_\theta$ and $U_{\theta+h}$. Observe that
\begin{align*}
&\lim_{h \to 0 } \frac{1}{h} \left[ \E\left( \int_{\tau_h\wedge t}^t \mathbb{A}_{\theta+h} f( x_{\theta+h}(s) , U_{\theta+h}(s) ) ds \right) - \E\left( \int_{\tau_h\wedge t}^t \mathbb{A}_{\theta} f( x_{\theta}(s) , U_{\theta}(s) ) ds \right) \right] \\
& = \sum_{i=0}^\infty \lim_{h \to 0 }   \frac{1}{h} \E\left[ \ind_{ \{ \sigma_i \wedge t \leq \tau_h < \sigma_{i+1} \wedge t   \} }  \int_{\tau_h\wedge t}^t \left( \mathbb{A}_{\theta+h} f( x_{\theta+h}(s) , U_{\theta+h}(s) ) \right. \right. \\ & \left. \left.  \qquad \qquad -  \mathbb{A}_{\theta} f( x_{\theta}(s) , U_{\theta}(s) )  \right) ds \right. \bigg]. 
\end{align*}
Recall the definition of $D_\theta \lambda_k(  x_\theta(t), U_\theta(t) ,\theta)$ from \eqref{defn_part_tod_der}. We shall soon prove that
\begin{align}
\label{prop:toproveadtertau}
& \lim_{h \to 0 }   \frac{1}{h} \E\left[ \ind_{ \{ \sigma_i \wedge t \leq \tau_h < \sigma_{i+1} \wedge t   \} }  \int_{\tau_h\wedge t}^t \left( \mathbb{A}_{\theta+h} f( x_{\theta+h}(s) , U_{\theta+h}(s) )\right. \right. \\ & \left. \left.  \qquad \qquad \qquad  -  \mathbb{A}_{\theta} f( x_{\theta}(s) , U_{\theta}(s) )  \right) ds  \right. \bigg]. \notag \\
& =  \sum_{k \in \mathcal{R}_d} \E \left[  \int_{ t \wedge \sigma_i}^{ t \wedge\sigma_{i +1} }  D_\theta \lambda_k(  x_\theta(s), U_\theta(s) ,\theta) \left(    \Delta_k \Psi_{t- s}( x_\theta(s), U_\theta(s), \theta ) )\right. \right. \notag \\ & \left. \left.  \qquad \qquad \qquad  -   \Delta_k f( x_\theta(s), U_\theta(s)) \right) ds  \right. \bigg]. \notag
\end{align}
Assuming this for now we get
\begin{align*}
&\lim_{h \to 0 } \frac{1}{h} \left[ \E\left( \int_{\tau_h\wedge t}^t \mathbb{A}_{\theta+h} f( x_{\theta+h}(s) , U_{\theta+h}(s) ) ds \right) - \E\left( \int_{\tau_h\wedge t}^t \mathbb{A}_{\theta} f( x_{\theta}(s) , U_{\theta}(s) ) ds \right) \right] \\
& = \sum_{k \in \mathcal{R}_d} \sum_{i=0}^\infty  \E \left[  \int_{t \wedge \sigma_i}^{ t \wedge \sigma_{i+1} }  D_\theta \lambda_k(  x_\theta(s), U_\theta(s) ,\theta) \left(    \Delta_k \Psi_{t- s}( x_\theta(s), U_\theta(s), \theta ) \right. \right. \notag \\ & \left. \left.  \qquad \qquad \qquad  -   \Delta_k f( x_\theta(s), U_\theta(s)) \right) ds  \right. \bigg] \\
& =  \sum_{k \in \mathcal{R}_d} \E \left[  \int_{0}^{ t}  \partial_\theta \lambda_k(  x_\theta(s), U_\theta(s) ,\theta) \Delta_k \Psi_{t- s}( x_\theta(s), U_\theta(s), \theta ) ds  \right] \\
& - \sum_{k \in \mathcal{R}_d} \E \left[  \int_{0}^{ t}  \partial_\theta \lambda_k(  x_\theta(s), U_\theta(s) ,\theta) \Delta_k f( x_\theta(s), U_\theta(s))  ds  \right] \\
&+  \sum_{k \in \mathcal{R}_d} \E \left[  \int_{0}^{ t}   \left\langle \nabla \lambda_k(x_\theta( s) ,U_\theta(s),\theta ) , y_\theta( s) \right\rangle   \Delta_k \Psi_{t- s}( x_\theta(s), U_\theta(s), \theta ) ds  \right] \\
& - \sum_{k \in \mathcal{R}_d} \E \left[  \int_{0}^{ t}  \left\langle \nabla \lambda_k(x_\theta( s) ,U_\theta(s),\theta ) , y_\theta( s) \right\rangle    \Delta_k f( x_\theta(s), U_\theta(s)) ds  \right] . 
\end{align*}
Combining this formula with \eqref{mainpropproof_part1} we obtain
\begin{align*}
& \hat{S}_\theta(f,t) \\
&= \lim_{h \to 0} \frac{ \E\left(  f(x_{\theta+h} (t), U_{\theta+h} (t)  ) \right)  - \E\left(   f(x_{\theta} (t), U_{\theta} (t) \right) }{h} \\
& =  \lim_{h \to 0} \frac{1}{h} \E\left[ \int_{0}^t \left( \mathbb{A}_{\theta+h} f( x_{\theta+h}(s) , U_{\theta+h}(s) )  -  \mathbb{A}_{\theta} f( x_{\theta}(s) , U_{\theta}(s) )  \right) ds \right] \\
& =  \lim_{h \to 0} \frac{1}{h} \E\left[ \int_{0}^{t \wedge \tau_h } \left( \mathbb{A}_{\theta+h} f( x_{\theta+h}(s) , U_{\theta+h}(s) )  -  \mathbb{A}_{\theta} f( x_{\theta}(s) , U_{\theta}(s) )  \right) ds \right]  \\
& + \lim_{h \to 0} \frac{1}{h} \E\left[ \int_{t \wedge \tau_h }^t \left( \mathbb{A}_{\theta+h} f( x_{\theta+h}(s) , U_{\theta+h}(s) )  -  \mathbb{A}_{\theta} f( x_{\theta}(s) , U_{\theta}(s) )  \right) ds \right]  \\
& = \sum_{k \in \mathcal{R}_c} \E\left[  \int_{0}^{t}  \partial_\theta \lambda_k  (x_\theta(s) ,U_\theta(s) ,\theta ) \left\langle \nabla f(x_\theta(s),U_\theta(s) ) , \zeta^{(c)}_k  \right\rangle ds   \right] \\
&+ \sum_{k \in \mathcal{R}_c} \E\left[  \int_{0}^{t}  \left \langle \nabla \left[   \lambda_k (x_\theta(s) ,U_\theta(s) ,\theta ) \left\langle \nabla f(x_\theta(s),U_\theta(s) ) , \zeta^{(c)}_k  \right\rangle \right] , y_\theta(s)   \right\rangle ds   \right]  \\
&+ \sum_{k \in \mathcal{R}_d} \E\left[  \int_{0}^{t}  \left \langle \nabla \left[   \lambda_k (x_\theta(s) ,U_\theta(s) ,\theta ) \Delta_k f(x_\theta(s),U_\theta(s)  )   \right] , y_\theta(s)   \right\rangle ds   \right] \\
& + \sum_{k \in \mathcal{R}_d} \E\left[  \int_{0}^{t} \partial_\theta \lambda_k  (x_\theta(s) ,U_\theta(s) ,\theta )  \Delta_k f(x_\theta(s) , U_\theta( s)  )  ds   \right] \\
&+  \sum_{k \in \mathcal{R}_d} \E \left[  \int_{ 0}^{ t } \partial_\theta \lambda_k  (x_\theta(s) ,U_\theta(s) ,\theta )   \Delta_k \Psi_{t- s}( x_\theta(s), U_\theta(s), \theta ) ds  \right] \\
&  + \sum_{k \in \mathcal{R}_d} \E \left[  \int_{ 0}^{ t }  \left\langle \nabla \lambda_k(x_\theta( s) ,U_\theta(s),\theta ) , y_\theta( s) \right\rangle  \Delta_k \Psi_{t- s}( x_\theta(s), U_\theta(s), \theta ) ds  \right] \\
&  - \sum_{k \in \mathcal{R}_d} \E \left[  \int_{ 0}^{ t } \partial_\theta \lambda_k  (x_\theta(s) ,U_\theta(s) ,\theta )  \Delta_k f(x_\theta(s) , U_\theta( s)  ) ds  \right]  \\
&  - \sum_{k \in \mathcal{R}_d} \E \left[  \int_{ 0}^{ t }  \left\langle \nabla \lambda_k(x_\theta( s) ,U_\theta(s),\theta ) , y_\theta( s) \right\rangle  \Delta_k f(x_\theta(s) , U_\theta( s)  ) ds  \right].
\end{align*}
In the last expression, the fourth term cancels with the sixth term. Expanding the third term via the product rule $\nabla (gh) = g \nabla h + h \nabla g$ produces two terms, one of which cancels with the last term and then we obtain the result stated in the statement of this proposition. Therefore to prove this proposition the only step remaining is to show \eqref{prop:toproveadtertau}. This is what we do next.

Assume that $x_\theta( \sigma_i ) = x$, $x_{\theta+h}( \sigma_i ) = x(h) =  x + o(1)$, $U_\theta(\sigma_i) =U_{\theta+h}(\sigma_i) =U$ and $\{\tau_h > \sigma_i\}$. Given this information $\mathcal{F}_i$, the random time $\delta_i = ( \tau_h -\sigma_i ) \wedge (\sigma_{i+1} -\sigma_i )$ has distribution that satisfies
\begin{align}
\label{distrideltai}
 \mathbb{P}\left(  \delta_i  \leq w  \vert \mathcal{F}_i  \right) = 1  - \exp\left(- \int_{0}^w \lambda_0( x_\theta(s + \sigma_i) , U ,\theta ) ds  \right) + o(1), \textnormal{ for } w \in [0, \infty)
\end{align}
where $\lambda_0(x,U,\theta) = \sum_{k \in \mathcal{R}_d} \lambda_k(x,U,\theta)$. Given $\delta_i  = w$, the probability that event \\$\{ (\sigma_{i+1} -\sigma_i ) >  ( \tau_h -\sigma_i ) \}$ occurs (i.e. $\delta_i = \tau_h -\sigma_i$) and the perturbation reaction is $k \in \mathcal{R}_d$ is simply
\begin{align*}
&\frac{1}{ \lambda_0(x_\theta( \sigma_i +w) , U,\theta  ) }\left| D_\theta \lambda_k(x_\theta(\sigma_i+w) , U,\theta ) \right| h +o(h).
\end{align*}
If $D_\theta \lambda_k(x_\theta(\sigma_i+w) , U,\theta ) > 0$ then at time $\tau_h$ process $U_{ \theta +h}$ jumps by $\zeta^{ (d) }_k$, and if $D_\theta \lambda_k(x_\theta(\sigma_i+w) , U,\theta ) <0$, process $U_{ \theta }$ jumps by $\zeta^{ (d) }_k$. We will suppose that the first situation holds but the other case can be handled similarly. Assuming $w < (t - \sigma_i)$ we have
 \begin{align*}
 & \lim_{h \to 0} \E\left( \int_{\tau_h\wedge t}^t \left( \mathbb{A}_{\theta+h} f( x_{\theta+h}(s) , U_{\theta+h}(s) )  -  \mathbb{A}_{\theta} f( x_{\theta}(s) , U_{\theta}(s) )  \right) ds  \middle\vert \mathcal{F}_i, \tau_h = \sigma_i + w ,k  \right)   \\
 & = \Delta_{k} \Psi_{t - \sigma_i -w} ( x_\theta( \sigma_i + w), U_\theta(  \sigma_i+w ), \theta ) -   \Delta_{k} f( x_\theta( \sigma_i + w), U_\theta(  \sigma_i+w ) ) \\
 & := G_k(  x_\theta( \sigma_i +w) , U_\theta( \sigma_i +w )  , t - \sigma_i-w)
\end{align*}
and as $\delta_i$ has distribution \eqref{distrideltai}, we obtain
\begin{align}
\label{mainpropproof_intstep1}
& \lim_{h \to 0} \frac{1}{h} \E\left[ \ind_{ \{ \sigma_i \wedge t \leq \tau_h < \sigma_{i+1} \wedge t   \} }  \int_{\tau_h\wedge t}^t \left( \mathbb{A}_{\theta+h} f( x_{\theta+h}(s) , U_{\theta+h}(s) )  -  \mathbb{A}_{\theta} f( x_{\theta}(s) , U_{\theta}(s) )  \right) ds \right] \notag \\
&= \sum_{k \in \mathcal{R}_d} \E \left[  \ind_{ \{ \sigma_i \leq t \}  } \int_{0}^{t - \sigma_i}  G_k(  x_\theta( \sigma_i +w) , U_\theta( \sigma_i +w )  , t - \sigma_i-w)    \right. \\  & \left. D_\theta \lambda_k(x_\theta(\sigma_i+w) , U_\theta(\sigma_i+w),\theta )  \exp \left( -\int_{0}^w \lambda_0(x_\theta( \sigma_i+s) , U_\theta(\sigma_i+s),\theta  ) ds \right)dw  \right]. \notag 
\end{align}

Note that given $\sigma_i < t$ and $\mathcal{F}_i$, the random variable $\gamma_i =  (t \wedge\sigma_{i +1} -  t \wedge\sigma_{i })$ has probability density function given by
\begin{align*}
p(w) =  \lambda_0(x_\theta( \sigma_i+ w ) , U_\theta(\sigma_i+w),\theta  )  \exp\left( -\int_{0}^w \lambda_0(x_\theta( \sigma_i+ u ) , U_\theta(\sigma_i+u),\theta  ) du  \right),
\end{align*}
for $w \in [0, t -\sigma_i)$ and  $ \mathbb{P}\left(  \gamma_i  \leq w  \vert \mathcal{F}_i  \right) =1$ if $w \geq (t - \sigma_i)$. Letting
\begin{align*} 
G(s,t)& = G_k(  x_\theta( s) , U_\theta( s)  , t - s) D_\theta \lambda_k(x_\theta(s) , U_\theta(s),\theta )\\
\quad \textnormal{and} \quad  P(w) & = \int_{w}^\infty p(u)du =   \exp\left( -\int_{0}^w \lambda_0(x_\theta( \sigma_i+ u ) , U_\theta(\sigma_i+u),\theta  ) du  \right)
\end{align*}
we have
\begin{align*}
&\E\left( \int_{t \wedge \sigma_i}^{ t \wedge \sigma_{i+1} }  G(s,t) ds \middle\vert \mathcal{F}_i , \sigma_i < t \right)  =  \E\left( \int_{0}^{ \gamma_i }  G(s+\sigma_i,t) ds \middle\vert \mathcal{F}_i , \sigma_i < t \right) \\
& =  \mathbb{P}\left( \gamma_i \geq t - \sigma_i \middle\vert \mathcal{F}_i , \sigma_i < t \right) \int_{0}^{t - \sigma_i}  G(s+\sigma_i,t) ds  \\
&+ \E\left( \ind_{ \{ 0 \leq \gamma_i <(t - \sigma_i) \} }  \int_{0}^{ \delta_i }  G(s+\sigma_i,t) ds \middle\vert \mathcal{F}_i , \sigma_i < t \right) \\
& = P(t - \sigma_i) \int_{0}^{t - \sigma_i}  G(s+\sigma_i,t) ds + \int_{0}^{t - \sigma_i} p(w) \left( \int_{0}^{ w }  G(s+\sigma_i,t) ds \right) dw. 
\end{align*}  
Using integration by parts 
\begin{align*}
\int_{0}^{t - \sigma_i} p(w) \left( \int_{0}^{ w }  G(s+\sigma_i,t) ds \right) dw &= -P(t - \sigma_i)\left( \int_{0}^{ t -\sigma_i }  G(s+\sigma_i,t) ds \right) \\
&+\int_{0}^{t - \sigma_i} P(w) G(w+\sigma_i,t) dw
\end{align*}
which shows that
\begin{align*}
 \int_{0}^{t - \sigma_i} P(w) G(w+\sigma_i,t) ds = \E\left( \int_{t \wedge \sigma_i}^{ t \wedge \sigma_{i+1} }  G(s,t) ds \middle\vert \mathcal{F}_i , \sigma_i < t \right) .
\end{align*}
Substituting this expression in \eqref{mainpropproof_intstep1} gives us
\begin{align*}
& \lim_{h \to 0} \frac{1}{h} \E\left[ \int_{\tau_h\wedge t}^t \left( \mathbb{A}_{\theta+h} f( x_{\theta+h}(s) , U_{\theta+h}(s) )  -  \mathbb{A}_{\theta} f( x_{\theta}(s) , U_{\theta}(s) )  \right) ds \right] \\
& = \sum_{i=0}^\infty \lim_{h \to 0} \frac{1}{h} \E\left[ \ind_{ \{ \sigma_i \wedge t \leq \tau_h < \sigma_{i+1} \wedge t   \} }  \int_{\tau_h\wedge t}^t \left( \mathbb{A}_{\theta+h} f( x_{\theta+h}(s) , U_{\theta+h}(s) ) \right. \right. \notag \\ & \left. \left.  \qquad \qquad \qquad  -  \mathbb{A}_{\theta} f( x_{\theta}(s) , U_{\theta}(s) )  \right) ds \right. \bigg]\\
&= \sum_{k \in \mathcal{R}_d} \sum_{i=0}^\infty \E \left[  \int_{ t \wedge \sigma_i}^{ t \wedge\sigma_{i +1} } G_k(  x_\theta( s) , U_\theta( s)  , t - s) D_\theta \lambda_k(x_\theta(s) , U_\theta(s),\theta ) ds  \right].
\end{align*}
This proves \eqref{prop:toproveadtertau} and completes the proof of this proposition.
\end{proof}

Define a $S_c \times S_c$ matrix by
\begin{align*}
M(x,U,\theta) =  \sum_{k \in \mathcal{R}_c  } \zeta^{ (c) }_{k}    (  \nabla \lambda_k(x,U,\theta) )^*
\end{align*}
for any $(x , U, \theta) \in \R^{S_c}  \times \N^{S_d}_0 \times \R$, where $v^*$ denotes the transpose of $v$. Let $\Phi(x_0,U_0,t)$ be the solution of the linear matrix-valued equations
\begin{align}
\label{defn:phixut}
\frac{d }{dt} \Phi(x_0,U_0,t) =   M(x_\theta(t) , U_\theta(t) ,\theta )  \Phi(x_0,U_0,t)
\end{align}
with $\Phi(x_0,U_0,0) = {\bf I}$, which is the $S_c \times S_c$ identity matrix. Here $(x_0, U_0)$ denotes the initial state of $(x_\theta(t) ,U_\theta(t) )$. It can be seen that $y_\theta(t)$, which is the solution of IVP \eqref{defn_y_thetat}, can be written as
\begin{align}
\label{defn_ythetat}
y_\theta(t) = \sum_{k \in \mathcal{R}_c  }  \int_0^t  \partial_\theta \lambda_k ( x_\theta(s) , U_\theta(s) , \theta)   \Phi(x_\theta (s),U_\theta(s),t - s) \zeta_k^{(c)}   ds.
\end{align}
This shall be useful in proving the next proposition which considers the sensitivity of $\Psi_t (x_\theta(t) , U_\theta(t) ,\theta )$ to the initial value of the continuous state $x_0$.

\begin{proposition}
\label{prop:secondprop}
Let $\Phi(x_0,U_0,t)$ be the matrix-valued function defined above. Then we can express the gradient of $\Psi_t(x_0,U_0,\theta) $ w.r.t. $x_0$ as
\begin{align}
\label{secondprop:mainidentity}
&\nabla \Psi_t(x_0,U_0,\theta) = \nabla f(x_0,U_0) \\
&+  \sum_{k \in \mathcal{R}_c} \E\left[  \int_{0}^{t}  \Phi^*(x_0, U_0,s)    \nabla \left[   \lambda_k (x_\theta(s) ,U_\theta(s) ,\theta ) \left \langle \nabla f(x_\theta(s),U_\theta(s) ) , \zeta^{(c)}_k \right\rangle  \right] ds   \right]  \notag  \\
&+ \sum_{k \in \mathcal{R}_d} \E\left[  \int_{0}^{t}  \lambda_k (x_\theta(s) ,U_\theta(s) ,\theta )  \Phi^*(x_0, U_0,s) \nabla \left( \Delta_k  f(x_\theta(s),U_\theta(s)  )    \right) ds   \right]  \notag  \\
&  + \sum_{k \in \mathcal{R}_d} \E \left[  \int_{ 0}^{ t }  \Phi^*(x_0, U_0,s) \nabla \lambda_k(x_\theta( s) ,U_\theta(s),\theta ) \Delta_k \Psi_{t-s}( x_\theta( s) , U_\theta( s ) ,\theta )  ds  \right].  \notag
\end{align}
\end{proposition}
\begin{proof}
To prove this proposition it suffices to show that for any vector $v \in \R^{S_c}$, the inner product of $v$ with the l.h.s. of \eqref{secondprop:mainidentity} is same as the inner product of $v$ with the r.h.s. of \eqref{secondprop:mainidentity}. Defining 
\begin{align*}
y(t) =   \Phi(x_0, U_0,t) v
\end{align*}
our aim is to prove that
\begin{align}
\label{secondprop:mainidentity2}
&\left\langle \nabla \Psi_t(x_0,U_0,\theta) , v \right \rangle = \left\langle \nabla f(x_0,U_0) ,  v  \right \rangle  \\
&+  \sum_{k \in \mathcal{R}_c} \E\left[  \int_{0}^{t}    \left\langle     \nabla \left[   \lambda_k (x_\theta(s) ,U_\theta(s) ,\theta ) \left \langle \nabla f(x_\theta(s),U_\theta(s) ) , \zeta^{(c)}_k \right\rangle  \right] , y(s)  \right \rangle  ds   \right]  \notag  \\
&+ \sum_{k \in \mathcal{R}_d} \E\left[  \int_{0}^{t}  \lambda_k (x_\theta(s) ,U_\theta(s) ,\theta )  \left \langle \nabla \left( \Delta_k  f(x_\theta(s),U_\theta(s)  ) , y(s) \right)  \right\rangle   ds   \right]  \notag  \\
&  + \sum_{k \in \mathcal{R}_d} \E \left[  \int_{ 0}^{ t }  \left \langle \nabla \lambda_k(x_\theta( s) ,U_\theta(s),\theta ) , y(s) \right\rangle  \Delta_k \Psi_{t-s}( x_\theta( s) , U_\theta( s ) ,\theta )  ds  \right].  \notag
\end{align}
Note that $y(t)$ solves the IVP
\begin{align}
\label{defn_y2}
\frac{d y}{dt} &= \sum_{k \in \mathcal{R}_c} \ \left\langle  \nabla  \lambda_k(x_\theta(t) , U_\theta(t) ,\theta ) , y(t)  \right\rangle  \zeta^{(c)}_k  \\
\textnormal{and}  & \qquad y(0) = v,  \notag
\end{align}
which shows that $y(t)$ is the directional derivative of $x_\theta(t)$ (see \eqref{defn_pdmp_param}) w.r.t. the initial state $x_0$ in the direction $v$.

This proposition can be proved in the same way as Proposition \ref{prop:mainprop}, by coupling process $(x_\theta, U_\theta)$ with another process $(x_{\theta,h} , U_{\theta,h} )$ according to
\begin{align*}
x_\theta(t) &= x_0 + \sum_{k \in \mathcal{R}_c} \left( \int_{0}^t \lambda_k( x_\theta(s), U_\theta(s), \theta ) ds \right) \zeta^{(c)}_k \\
x_{\theta,h}(t) &= x_0 + h v + \sum_{k \in \mathcal{R}_c} \left( \int_{0}^t \lambda_k( x_{\theta, h}(s), U_{\theta, h}(s), \theta ) ds \right) \zeta^{(c)}_k \\
U_\theta(t) & = U_0 + \sum_{k \in \mathcal{R}_d} Y_k \left( \int_{0}^t \lambda_k( x_\theta(s), U_\theta(s) ,\theta ) \wedge \lambda_k( x_{\theta, h}(s), U_{\theta,h}(s), \theta )   ds  \right) \zeta^{(d)}_k \\
& +  \sum_{k \in \mathcal{R}_d} Y^{(1)}_k \left( \int_{0}^t \lambda^{(1)}_k( x_\theta(s), U_\theta(s) ,\theta , x_{\theta, h}(s), U_{\theta, h}(s), \theta )   ds  \right) \zeta^{(d)}_k \\
U_{\theta,h}(t) & = U_0 + \sum_{k \in \mathcal{R}_d} Y_k \left( \int_{0}^t \lambda_k( x_\theta(s), U_\theta(s) ,\theta ) \wedge \lambda_k( x_{\theta, h}(s), U_{\theta, h}(s), \theta )   ds  \right) \zeta^{(d)}_k \\
& +  \sum_{k \in \mathcal{R}_d} Y^{(2)}_k \left( \int_{0}^t \lambda^{(2)}_k( x_\theta(s), U_\theta(s) ,\theta , x_{\theta , h}(s), U_{\theta , h}(s), \theta  ) ds  \right) \zeta^{(d)}_k, 
\end{align*}
where $\{ Y_k , Y^{ (1) }_k , Y^{ (2)}_{k}\}$ is a collection of independent unit-rate Poisson processes, and $\lambda^{(1)}_k$, $\lambda^{(2)}_k$ are as in the proof of Proposition \ref{prop:mainprop}. An important difference between this proposition and Proposition \ref{prop:mainprop}, is that the value of $\theta$ is the same in the coupled processes, and hence the only difference between the two processes comes due to difference in the initial continuous state $x_0$. Consequently the $ \partial_\theta \lambda_k$ terms in the statement of Proposition \ref{prop:mainprop} \emph{disappear} and we obtain \eqref{secondprop:mainidentity2}. 
\end{proof}

\begin{proof}[Proof of Theorem \ref{thm:convergence}]
Define
\begin{align*}
L(t) = \sum_{k \in  \mathcal{R}_c} \E\left[  \int_{0}^t   \partial_\theta \lambda_k   (x_\theta(s) ,U_\theta(s) ,\theta ) \left\langle  \nabla \Psi_{t-s}(x_\theta(s),U_\theta(s),t-s) , \zeta^{(c)}_k  \right\rangle ds \right]. 
\end{align*}
Due to Proposition \ref{prop:mainprop}, to prove Theorem \ref{thm:convergence} it suffices to prove that
\begin{align}
\label{cond_suff_mainthem}
 L(T) & = \sum_{k \in \mathcal{R}_c} \E\left[   \int_{0}^T   \partial_\theta \lambda_k  (x_\theta(t) ,U_\theta(t) ,\theta ) \left\langle  \nabla  f(x_\theta(t) ,U_\theta(t) ) , \zeta^{(c)}_k  \right\rangle dt  \right] \\
&+\sum_{k \in \mathcal{R}_c} \E\left[  \int_{0}^T \left\langle    \nabla \left[   \lambda_k (x_\theta(t) ,U_\theta(t) ,\theta ) \langle \nabla  f(x_\theta(t),U_\theta(t) ) , \zeta^{(c)}_k  \rangle \right] , y_\theta(t) \right\rangle dt  \right] \notag \\
& + \sum_{k \in \mathcal{R}_d} \E\left[  \int_{0}^T  \lambda_k (x_\theta(t) ,U_\theta(t) ,\theta ) \left\langle    \nabla    \left(  \Delta_k f(x_\theta(t),U_\theta(t)  )    \right)  , y_\theta(t) \right\rangle dt  \right] \notag \\
&+\sum_{k \in \mathcal{R}_d}  \E\left[  \int_{0}^T  \left\langle \nabla \lambda_k(x_\theta( t) ,U_\theta(t),\theta ) ,y_\theta(t)  \right\rangle   \Delta_k \Psi_{T- t}( x_\theta(t), U_\theta(t), \theta ) dt \right]. \notag
\end{align}

Let $\{ \mathcal{F}_t \}$ be the filtration generated by process $(x_\theta, U_\theta)$. For any $t \geq 0$ let $\E_t (\cdot )$ denote the conditional expectation $\E( \cdot \vert \mathcal{F}_t )$. Proposition \ref{prop:secondprop} allows us to write
\begin{align*}
&\nabla \Psi_{t - s}(x_\theta(s) , U_\theta(s), t-s ) \\
& = \nabla f(x_\theta(s) , U_\theta(s) ) \\
&+  \sum_{k \in \mathcal{R}_c}  \int_{s}^{t} \E_s \left. \Big[  \Phi^*(x_\theta(s), U_\theta(s),u-s)  \right. \\& \left. \qquad \qquad  \nabla \left[   \lambda_k (x_\theta(u) ,U_\theta(u) ,\theta ) \left \langle \nabla f(x_\theta(u),U_\theta(u) ) , \zeta^{(c)}_k \right\rangle  \right]  \right]  du   \notag  \\
&+ \sum_{k \in \mathcal{R}_d}  \int_{s}^{t} \E_s \left[  \lambda_k (x_\theta(u) ,U_\theta(u) ,\theta )  \Phi^*(x_\theta(s), U_\theta(s),u-s) \nabla \left( \Delta_k  f(x_\theta(u),U_\theta(u)  )    \right)  \right] du    \notag  \\
&  + \sum_{k \in \mathcal{R}_d} \int_{ s}^{ t } \E_s \left[   \Phi^*(x_\theta(s), U_\theta(s), u -s)  \right. \\& \left. \qquad \qquad  \nabla \lambda_k(x_\theta( u) ,U_\theta(u),\theta ) \Delta_k \Psi_{t-u}( x_\theta( u) , U_\theta( u ) ,\theta ) \right]  du  . 
\end{align*}
This shows that
\begin{align*}
& \frac{d}{dt} \nabla \Psi_{t - s}(x_\theta(s) , U_\theta(s), t-s )\\
  &=  \sum_{k \in \mathcal{R}_c}  \E_s \left[  \Phi^*(x_\theta(s), U_\theta(s),t-s)    \nabla \left[   \lambda_k (x_\theta(t) ,U_\theta(t) ,\theta ) \left \langle \nabla f(x_\theta(t),U_\theta(t) ) , \zeta^{(c)}_k \right\rangle  \right]  \right]   \notag  \\
&+ \sum_{k \in \mathcal{R}_d}  \E_s \left[  \lambda_k (x_\theta(t) ,U_\theta(t) ,\theta )  \Phi^*(x_\theta(s), U_\theta(s),t-s) \nabla \left( \Delta_k  f(x_\theta(t),U_\theta(t)  )    \right)  \right]     \notag  \\
&  + \sum_{k \in \mathcal{R}_d} \E_s \left[   \Phi^*(x_\theta(s), U_\theta(s), t -s) \nabla \lambda_k(x_\theta( t) ,U_\theta(t),\theta ) \Delta_k f( x_\theta( t) , U_\theta(t )  ) \right]  \\  
&  + \sum_{k \in \mathcal{R}_d} \int_{ s}^{ t } \E_s \left. \bigg[   \Phi^*(x_\theta(s), U_\theta(s), u -s)  \right. \\& \left. \qquad \qquad  \nabla \lambda_k(x_\theta( u) ,U_\theta(u),\theta ) \frac{d}{dt}  \Delta_k \Psi_{t-u}( x_\theta( u) , U_\theta( u ) ,\theta ) \right]  du.  
\end{align*}
The middle two terms can be combined using the product rule $\nabla (gh) = g \nabla h + h \nabla g$ to yield
\begin{align*}
& \frac{d}{dt} \nabla \Psi_{t - s}(x_\theta(s) , U_\theta(s), t-s )\\
  &=  \sum_{k \in \mathcal{R}_c}  \E_s \left[  \Phi^*(x_\theta(s), U_\theta(s),t-s)    \nabla \left[   \lambda_k (x_\theta(t) ,U_\theta(t) ,\theta ) \left \langle \nabla f(x_\theta(t),U_\theta(t) ) , \zeta^{(c)}_k \right\rangle  \right]  \right]   \notag  \\
&+ \sum_{k \in \mathcal{R}_d}  \E_s \left[  \Phi^*(x_\theta(s), U_\theta(s),t-s) \nabla \left(  \lambda_k (x_\theta(t) ,U_\theta(t) ,\theta )  \Delta_k  f(x_\theta(t),U_\theta(t)  )    \right)  \right]     \notag  \\
&  + \sum_{k \in \mathcal{R}_d} \int_{ s}^{ t } \E_s \left. \bigg[   \Phi^*(x_\theta(s), U_\theta(s), u -s)  \right. \\& \left. \qquad \qquad  \nabla \lambda_k(x_\theta( u) ,U_\theta(u),\theta ) \frac{d}{dt}  \Delta_k \Psi_{t-u}( x_\theta( u) , U_\theta( u ) ,\theta ) \right]  du.   \notag
\end{align*}
Using this we can compute the time-derivative of $L(t)$ as
\begin{align}
\label{LTimederformula}
\frac{d L(t) }{dt} = \sum_{k \in \mathcal{R}_c}  \E\left[  \partial_\theta \lambda_k  (x_\theta(t) ,U_\theta(t) ,\theta ) \left\langle  \nabla  f(x_\theta(t) ,U_\theta(t) ) , \zeta^{(c)}_k  \right\rangle \right] + A+ B+ C,
\end{align}
where   
\begin{align*}
&A  :=   \sum_{k \in \mathcal{R}_c} \sum_{j \in \mathcal{R}_c}  \int_{0}^{t}  \E\left[   \partial_\theta \lambda_k  (x_\theta(s) ,U_\theta(s) ,\theta )  \left. \Big\langle \Phi^*(x_\theta( s) , U_\theta( s) , t-s) \right. \right. \\& \left. \left. \qquad \qquad    \nabla \left[   \lambda_j (x_\theta(t) ,U_\theta(t) ,\theta ) \langle \nabla f(x_\theta(t),U_\theta(t) ) , \zeta^{(c)}_j  \rangle \right], \zeta^{(c)}_k  \right\rangle \right] ds, \\
&B:=  \sum_{k \in \mathcal{R}_c} \sum_{j \in \mathcal{R}_d}    \int_{0}^{t}   \E\left[    \partial_\theta \lambda_k  (x_\theta(s) ,U_\theta(s) ,\theta ) \right. \\ 
&  \left. \left\langle \Phi^*(x_\theta( s) , U_\theta( s) , t-s)      \nabla  \left[  \lambda_j (x_\theta(t) ,U_\theta(t) ,\theta )  \Delta_j  f(x_\theta(t),U_\theta(t)  )    \right] , \zeta^{(c)}_k  \right\rangle  \right] ds \\
\textnormal{and} \\
&C: =   \sum_{k \in \mathcal{R}_c} \sum_{j \in \mathcal{R}_d}  \int_{0}^{t}   \E\left[    \partial_\theta \lambda_k  (x_\theta(s) ,U_\theta(s) ,\theta )\left\langle  \int_{ s}^{ t }  \Phi^*(x_\theta(s), U_\theta(s), u -s)  \right. \right. \\& \left. \left. \qquad \qquad    \nabla \lambda_j(x_\theta( u) ,U_\theta(u),\theta ) \frac{d}{dt}  \Delta_j \Psi_{t-u}( x_\theta( u) , U_\theta( u ) ,\theta ) du, \zeta^{(c)}_k  \right\rangle  \right] ds. 
\end{align*} 
This definition of $A$, $B$ and $C$ ensures that
\begin{align*}
&A+ B+ C \\
& = \sum_{k \in \mathcal{R}_c} \int_{0}^{t}  \E\left[   \partial_\theta \lambda_k  (x_\theta(s) ,U_\theta(s) ,\theta )  \left\langle \frac{d}{dt} \nabla \Psi_{t - s}(x_\theta(s) , U_\theta(s), t-s ), \zeta^{(c)}_k  \right\rangle ds \right].
\end{align*}

Recall that $y_\theta(t)$ can be expressed as \eqref{defn_ythetat}. Therefore we can write $A$ as
\begin{align}
\label{Ainproperform}
A  &=  \sum_{j \in \mathcal{R}_c}  \E\left[  \left\langle  \nabla \left[   \lambda_j (x_\theta(t) ,U_\theta(t) ,\theta ) \langle \nabla f(x_\theta(t),U_\theta(t) ) , \zeta^{(c)}_j  \rangle \right] ,  \right.\right. \notag \\ 
& \quad \quad \left.\left.   \sum_{k \in \mathcal{R}_c}   \int_{0}^{t}  \partial_\theta \lambda_k  (x_\theta(s) ,U_\theta(s) ,\theta )  \Phi(x_\theta( s) , U_\theta( s) , t-s)   \zeta^{(c)}_k  \right\rangle ds \right]  \notag \\
& =  \sum_{j \in \mathcal{R}_c}  \E\left[  \left\langle  \nabla \left[   \lambda_j (x_\theta(t) ,U_\theta(t) ,\theta ) \langle \nabla f(x_\theta(t),U_\theta(t) ) , \zeta^{(c)}_j  \rangle \right] , y_\theta(t)   \right\rangle  \right]. 
\end{align}
Similarly we can write $B$ as
\begin{align}
\label{Binproperform}
 B =  \sum_{j \in \mathcal{R}_d}  \E\left[   \left\langle   \nabla \left[   \lambda_j (x_\theta(t) ,U_\theta(t) ,\theta )  \Delta_j f(x_\theta(t),U_\theta(t)  )    \right] , y_\theta(t) \right\rangle \right].
\end{align}

Changing the order of integration we can write $C$ as
\begin{align*}
&C=   \sum_{j \in \mathcal{R}_d}  \int_{0}^{t}   \E\left[     \left\langle  \nabla \lambda_j(x_\theta( u) ,U_\theta(u),\theta ) \frac{d}{dt}  \Delta_j \Psi_{t-u}( x_\theta( u) , U_\theta( u ) ,\theta ) ,  \right. \right. \\ 
&\quad \quad  \left. \left. \sum_{k \in \mathcal{R}_c}   \int_{ 0}^{ u} \partial_\theta \lambda_k  (x_\theta(s) ,U_\theta(s) ,\theta )    \Phi(x_\theta(s), U_\theta(s), u -s)\zeta^{(c)}_k   ds \right\rangle du  \right]   \\
& = \sum_{j \in \mathcal{R}_d}  \int_{0}^{t}   \E\left[     \left\langle  \nabla \lambda_j(x_\theta( u) ,U_\theta(u),\theta ) \frac{d}{dt}  \Delta_j \Psi_{t-u}( x_\theta( u) , U_\theta( u ) ,\theta ) , y_\theta(u)\right\rangle du \right]\\
& = \sum_{j \in \mathcal{R}_d}  \frac{d}{dt}    \int_{0}^{t}   \E\left[     \left\langle  \nabla \lambda_j(x_\theta( u) ,U_\theta(u),\theta )\Delta_j \Psi_{t-u}( x_\theta( u) , U_\theta( u ) ,\theta ) , y_\theta(u)\right\rangle du \right] \\
&  -  \sum_{j \in \mathcal{R}_d}   \E\left[     \left\langle  \nabla \lambda_j(x_\theta( t) ,U_\theta(t),\theta )\Delta_j f ( x_\theta( t) , U_\theta( t ) ) , y_\theta(t)\right\rangle  \right] .
\end{align*}
This relation along with \eqref{Ainproperform}, \eqref{Binproperform} and \eqref{LTimederformula} implies that  
\begin{align*}
\frac{d L(t) }{dt} &= \sum_{k \in \mathcal{R}_c} \E\left[    \partial_\theta \lambda_k (x_\theta(t) ,U_\theta(t) ,\theta ) \left\langle  \nabla f(x_\theta(t) ,U_\theta(t) ) , \zeta^{(c)}_k  \right\rangle \right]  \\
& +\sum_{k \in \mathcal{R}_c}  \E\left[   \left\langle   \nabla \left[   \lambda_k (x_\theta(t) ,U_\theta(t) ,\theta ) \langle \nabla f(x_\theta(t),U_\theta(t) ) , \zeta^{(c)}_k  \rangle \right] , y_\theta(t) \right\rangle  \right] \\
&+  \sum_{k \in \mathcal{R}_d}\E\left[    \left\langle   \nabla \left[   \lambda_k (x_\theta(t) ,U_\theta(t) ,\theta )  \Delta_k f(x_\theta(t),U_\theta(t)  )     \right] , y_\theta(t) \right\rangle  \right]\\
&+ \sum_{k \in \mathcal{R}_d}  \frac{d}{dt}  \int_{0}^t  \E\left[  \left\langle \nabla \lambda_k(x_\theta( s) ,U_\theta(s),\theta )    \Delta_k \Psi_{t-s}( x_\theta( s) , U_\theta( s ) ,\theta ) ,y_\theta(s)  \right\rangle ds \right] \\
& -  \sum_{k \in \mathcal{R}_d}   \E\left[     \left\langle  \nabla \lambda_k(x_\theta( t) ,U_\theta(t),\theta )\Delta_k f ( x_\theta( t) , U_\theta( t ) ) , y_\theta(t)\right\rangle  \right].
\end{align*}
Applying the product rule on the third term, will produce two terms, one of which will cancel with the last term to yield
\begin{align*}
\frac{d L(t) }{dt} &= \sum_{k \in \mathcal{R}_c} \E\left[    \partial_\theta \lambda_k (x_\theta(t) ,U_\theta(t) ,\theta ) \left\langle  \nabla f(x_\theta(t) ,U_\theta(t) ) , \zeta^{(c)}_k  \right\rangle \right]  \\
& +\sum_{k \in \mathcal{R}_c}  \E\left[   \left\langle   \nabla \left[   \lambda_k (x_\theta(t) ,U_\theta(t) ,\theta ) \langle \nabla f(x_\theta(t),U_\theta(t) ) , \zeta^{(c)}_k  \rangle \right] , y_\theta(t) \right\rangle  \right] \\
&+  \sum_{k \in \mathcal{R}_d}\E\left[   \lambda_k (x_\theta(t) ,U_\theta(t) ,\theta )   \left\langle   \nabla \left(   \Delta_k f(x_\theta(t),U_\theta(t)  )     \right) , y_\theta(t) \right\rangle  \right]\\
&+ \sum_{k \in \mathcal{R}_d}  \frac{d}{dt}  \int_{0}^t  \E\left[  \left\langle \nabla \lambda_k(x_\theta( s) ,U_\theta(s),\theta )    \Delta_k \Psi_{t-s}( x_\theta( s) , U_\theta( s ) ,\theta ) ,y_\theta(s)  \right\rangle ds \right] .
\end{align*}
Integrating this equation from $t = 0$ to $t =T$ will prove \eqref{cond_suff_mainthem} and this completes the proof of Theorem \ref{thm:convergence}.
\end{proof}

\begin{proof}[Proof of Theorem \ref{thm:representation}]
Consider the Markov process $( x_\theta(t) , U_\theta(t) , y_\theta(t)  )_{t \geq 0}$. The generator of this process is given by
\begin{align*}
\mathbb{H} F(x,u,y) &= \sum_{k \in  \mathcal{R}_c} \lambda_k(x,u,\theta) \left\langle \nabla F(x,u,y), \zeta^{(c)}_k  \right\rangle +    \sum_{k \in  \mathcal{R}_d} \lambda_k(x,u,\theta) \Delta_k F(x,u,y) \\
&+ \sum_{k \in  \mathcal{R}_c} \partial_\theta \lambda_k(x,u,\theta) \left\langle \nabla_y F(x,u,y), \zeta^{(c)}_k  \right\rangle \\
&+ \sum_{k \in  \mathcal{R}_c}  \left\langle \nabla  \lambda_k(x,u,\theta) , y \right\rangle \left\langle \nabla_y F(x,u,y), \zeta^{(c)}_k  \right\rangle
\end{align*}
for any real-valued function $F: \R^{S_c} \times \N^{S_d}_0 \times \R^{S_c} \to \R$. Here $\nabla_y F$ denotes the gradient of function $F$ w.r.t.\ the last $S_c$ coordinates. Setting 
\begin{align*}
F(x,u,y) = \left\langle  \nabla f(x,u), y  \right\rangle
\end{align*}
we obtain
\begin{align*}
\mathbb{H} F(x,u,y) &= \sum_{k \in  \mathcal{R}_c} \lambda_k(x,u,\theta) \left\langle \Delta f(x,u) y, \zeta^{(c)}_k  \right\rangle +    \sum_{k \in  \mathcal{R}_d} \lambda_k(x,u,\theta) \Delta_k  \left\langle  \nabla f(x,u), y  \right\rangle \\
&+ \sum_{k \in  \mathcal{R}_c} \partial_\theta \lambda_k(x,u,\theta) \left\langle \nabla f(x,u), \zeta^{(c)}_k  \right\rangle \\
& + \sum_{k \in  \mathcal{R}_c}  \left\langle \nabla  \lambda_k(x,u,\theta) , y \right\rangle \left\langle \nabla f(x,u), \zeta^{(c)}_k  \right\rangle
\end{align*}
where $\Delta F$ denotes the Hessian matrix of $F$ w.r.t.\ the first $S_c$ coordinates. However note that the first and the fourth terms can be combined with product-rule as
\begin{align*}
&\sum_{k \in  \mathcal{R}_c} \lambda_k(x,u,\theta) \left\langle \Delta f(x,u) y, \zeta^{(c)}_k  \right\rangle  +\sum_{k \in  \mathcal{R}_c}  \left\langle \nabla  \lambda_k(x,u,\theta) , y \right\rangle \left\langle \nabla f(x,u), \zeta^{(c)}_k  \right\rangle \\
& = \sum_{k \in  \mathcal{R}_c} \left\langle    \nabla \left[   \lambda_k (x, u ,\theta ) \langle \nabla  f(x,u ) , \zeta^{(c)}_k  \rangle \right] , y \right\rangle
\end{align*}
and hence we get
\begin{align}
\label{genH_simpl}
\mathbb{H} F(x,u,y) & =  \sum_{k \in  \mathcal{R}_c} \partial_\theta \lambda_k(x,u,\theta) \left\langle \nabla f(x,u), \zeta^{(c)}_k  \right\rangle \\ 
&+  \sum_{k \in  \mathcal{R}_c} \left\langle    \nabla \left[   \lambda_k (x, u ,\theta ) \langle \nabla  f(x,u ) , \zeta^{(c)}_k  \rangle \right] , y \right\rangle \notag \\
&+ \sum_{k \in  \mathcal{R}_d} \lambda_k(x,u,\theta) \Delta_k  \left\langle  \nabla f(x,u), y  \right\rangle.  \notag
\end{align}
Using Dynkin's formula we have
\begin{align*}
\E\left( F(x_\theta(T) , U_\theta(T) , y_\theta(T)  ) \right)= \E \left[  \int_0^T  \mathbb{H} F(x_\theta(t) , U_\theta(t) , y_\theta(t))dt \right]
\end{align*}
and substituting \eqref{genH_simpl} yields
\begin{align*}
& \E\left[  \left \langle \nabla f (x_\theta(T) ,U_\theta(T) ) , y_\theta(T)   \right\rangle \right]\\
& = \sum_{k \in \mathcal{R}_c} \E\left[   \int_{0}^T   \partial_\theta \lambda_k  (x_\theta(t) ,U_\theta(t) ,\theta ) \left\langle  \nabla  f(x_\theta(t) ,U_\theta(t) ) , \zeta^{(c)}_k  \right\rangle dt  \right] \\
&+\sum_{k \in \mathcal{R}_c} \E\left[  \int_{0}^T \left\langle    \nabla \left[   \lambda_k (x_\theta(t) ,U_\theta(t) ,\theta ) \langle \nabla  f(x_\theta(t),U_\theta(t) ) , \zeta^{(c)}_k  \rangle \right] , y_\theta(t) \right\rangle dt  \right] \\
& + \sum_{k \in \mathcal{R}_d} \E\left[  \int_{0}^T  \lambda_k (x_\theta(s) ,U_\theta(s) ,\theta ) \left\langle    \nabla    \left(  \Delta_k f(x_\theta(t),U_\theta(t)  )    \right)  , y_\theta(t) \right\rangle dt  \right].\\
\end{align*}
This relation along with Proposition \ref{prop:mainprop} proves Theorem \ref{thm:representation}.
\end{proof}

\section*{Acknowledgments}
This work was funded by the European Research Council (ERC) under the European Union's Horizon 2020 research and innovation programme (grant agreement 743269).

\bibliographystyle{unsrt}

\end{document}